\documentclass[10pt]{amsart}

\usepackage[colorlinks,citecolor=blue,urlcolor=blue]{hyperref}
\usepackage[toc,page,titletoc]{appendix}
\usepackage{amssymb}

\usepackage[version=4]{mhchem}
\usepackage{graphicx}
\usepackage{url}
\usepackage{bbm}
\usepackage{mathtools}
\newcounter{tmp}
\allowdisplaybreaks

\usepackage{tikz-cd}

\numberwithin{equation}{section}
\theoremstyle{plain}
\newtheorem{theorem}{Theorem}[section]
\newtheorem{lemma}[theorem]{Lemma}
\newtheorem{corollary}[theorem]{Corollary}
\newtheorem{proposition}[theorem]{Proposition}

\theoremstyle{definition}
\newtheorem{definition}[theorem]{Definition}
\newtheorem{example}[theorem]{Example}
\theoremstyle{remark}
\newtheorem{remark}[theorem]{Remark}
\newcommand{\cA}{\mathcal{A}}
\newcommand{\cE}{\mathcal{E}}
\newcommand{\cF}{\mathcal{F}}
\newcommand{\cG}{\mathcal{G}}
\newcommand{\cI}{I}
\newcommand{\cK}{\mathcal{K}}

\newcommand{\cU}{\mathcal{U}}
\newcommand{\cV}{\mathcal{V}}
\newcommand{\supp}{{\rm supp\,}}
\newcommand{\spa}{{\rm span\,}}

\newcommand{\leftrightharpoonup}{%
  \mathrel{\mathpalette\lrhup\relax}%
}
\newcommand{\lrhup}[2]{%
  \ooalign{$#1\leftharpoonup$\cr$#1\rightharpoonup$\cr}%
}

\makeatletter
\newcommand\RSloop{\@ifnextchar\bgroup\RSloopa\RSloopb}
\makeatother
\newcommand\RSloopa[1]{\bgroup\RSloop#1\relax\egroup\RSloop}
\newcommand\RSloopb[1]%
  {\ifx\relax#1%
   \else
     \ifcsname RS:#1\endcsname
       \csname RS:#1\endcsname
     \else
       \GenericError{(RS)}{RS Error: operator #1 undefined}{}{}%
     \fi
   \expandafter\RSloop
   \fi
  }
\newcommand\X{0}
\newcommand\RS[1]%
  {\begin{tikzpicture}
     [every node/.style=
       {circle,draw,fill,minimum size=1.5pt,inner sep=0pt,outer sep=0pt},
      line cap=round
     ]
   \coordinate(\X) at (0,0);
   \RSloop{#1}\relax
   \end{tikzpicture}
  }
\newcommand\RSdef[1]{\expandafter\def\csname RS:#1\endcsname}
\newlength\RSu
\RSdef{i}{\draw (\X) -- +(90:\RSu) node{};}
\RSdef{l}{\draw (\X) -- +(135:\RSu) node{};}
\RSdef{r}{\draw (\X) -- +(45:\RSu) node{};}
\RSdef{I}{\draw (\X) -- +(90:\RSu) coordinate(\X I);\edef\X{\X I}}
\RSdef{L}{\draw (\X) -- +(135:\RSu) coordinate(\X L);\edef\X{\X L}}
\RSdef{R}{\draw (\X) -- +(45:\RSu) coordinate(\X R);\edef\X{\X R}}

\allowdisplaybreaks

\newcommand{\parallelsum}{\mathbin{\!/\mkern-5mu/\!}}

\ExplSyntaxOn
\keys_define:nn { mhchem }
 {
  arrow-min-length .code:n =
   \cs_set:Npn \__mhchem_arrow_options_minLength:n { {#1} } % default is 2em
 }
\ExplSyntaxOff
\mhchemoptions{arrow-min-length=1em}

\begin{document}
\title
  {First order endotactic reaction networks}
\author{Chuang Xu}\thanks{Department of Mathematics, University of Hawai'i at M\={a}noa, Honolulu, Hawai'i, 96822, USA. Email address: chuangxu@hawaii.edu. This work is supported by a start-up grant from the University of Hawai'i at M\={a}noa and a Travel Support for Mathematicians from the Simons Foundation (MP-TSM-00002379).}


\noindent

\begin{abstract}
Reaction networks are a general framework widely used in modeling diverse phenomena in different science disciplines. The dynamical process of a reaction network endowed with \emph{mass-action} kinetics is a mass-action system as an ODE defined by a directed graph, the so-called ``reaction graph''. Endotacticity is a graph property used to study persistence and permanence of mass-action systems. In this paper, we provide a detailed characterization of first order endotactic reaction graphs. Besides, we provide a sufficient condition for endotacticity of reaction networks which are \emph{not} necessarily of first order. 

Such a characterization of a first order endotactic reaction graph yields the spectral property of the adjacency matrix of the reaction graph. As a consequence, we prove that every first order \emph{endotactic} mass-action system as a linear ODE has a weakly reversible deficiency zero realization, and has a unique equilibrium which is exponentially globally asymptotically stable (and is positive) in each (positive) stoichiometric compatibility class. Using a stability result for \emph{asymptotically autonomous differential equations}, examples are constructed to illustrate that the global stability results can be extended to mass-action systems of higher order reaction networks modeled by \emph{nonlinear} ODEs, which are \emph{not} necessarily endotactic. 
{Different from the classical approaches for proving global asymptotic stability, the proof does not rely on the construction of a Lyapunov function.} 
This paper may serve as a starting point of characterizing endotactic reaction graphs of higher orders and studying global stability of mass-action systems in general.
\end{abstract}


\maketitle

\section{Introduction}

\begingroup
\setcounter{tmp}{\value{theorem}}% store current value of theorem counter
\setcounter{theorem}{0} %assign desired value to theorem counter
\renewcommand\thetheorem{\Alph{theorem}}% locally redefine the representation of the theorem counter

{\emph{\textbf{Given an autonomous ordinary differential equation (ODE) defined by a directed graph, when does the ODE admit a globally asymptotically stable equilibrium?}} {This general question is motivated by a well-known Global Attractor Conjecture in Chemical Reaction Network Theory (CRNT). This paper identifies a wide class of \emph{linear} autonomous ODEs which admits a globally asymptotically stable equilibrium, and the associated directed graph is the so-called \emph{first order 
endotactic reaction graph}. Instead of relying on the standard approach of constructing Lyapunov functions but rather by characterizing the spectral property of first order endotactic reaction graphs purely on the graph property\---\emph{endotacticity} and \emph{first order}, without knowing apriori the weights of the graph, we are able to prove that the linear ODEs admit a globally asymptotically stable equilibrium.}

 Reaction networks are widely used as a modeling regime in diverse science fields, including molecular biology \cite{G03}, computer science \cite{SCWB08},  genetics \cite{B45}, game theory \cite{VRDF14}, and social sciences \cite{DW05}. CRNT has become a live research area on the study of reaction networks from different perspectives, since the pioneering works on the mathematics of reaction networks by Feinberg, Horn, Jackson et al \cite{F72,H72,HJ72,CF05}. 
 
 A reaction network consists of a triple set, of \emph{species}, \emph{complexes} and \emph{reactions}, and can be represented by a directed graph, called the ``reaction graph''  (as an embedded graph in a Euclidean space).} 
{When species of a reaction network are \emph{in abundance}, instead of counting the number of molecules, one considers the \emph{concentration} of species as a \emph{mean field approximation} of the fraction of species counts over a diverging volume  \cite{K70,K71}, and the evolution of concentration of species is governed by an ODE \cite{F19}. A reaction network modeled by such an ODE is called a \emph{(deterministic) (reaction) system}.} Below we provide a brief review of the Global Attractor Conjecture.
\smallskip

\noindent\textbf{Review of GAC}. {A biochemically interesting class of reaction systems are \emph{complex-balanced} systems  \cite{F19} (see \eqref{eq:complex-balancing} in Section~\ref{sect:GAS} for its definition). It is known that a complex-balanced \emph{mass-action} system (see Section~\ref{subsect:kinetics} for the definition of mass-action) has a unique positive equilibrium in each \emph{positive stoichiometric compatibility class}  (roughly speaking, an open forward invariant subset of the ODE; see Section~\ref{sect:RN} for its definition) \cite{F87,G03}. It was conjectured that
\emph{the unique positive equilibrium (the so-called ``Birch point'' \cite{CDSS09}) of a complex-balanced mass-action system is globally attractive  in  each positive stoichiometric compatibility class}. This conjecture is referred to as the {Global Attractor Conjecture} (GAC) \cite{HJ72}. In the light of that the $\omega$-limit set of a complex-balanced mass-action system consists of only equilibria \cite{F87,SM00,S01} which implies all trajectories of the mass-action system are bounded, GAC can be rephrased as \emph{every complex-balanced mass-action system is persistent} \cite{CNP13,GMS14}, which means all trajectories of the system keep a \emph{non-vanishing distance} from the boundary of the positive cone $\mathbb{R}^d_+$ of the Euclidean space.}

{To study this conjecture, 
Craciun, Nazarov, and Pantea introduced a concept for reaction graphs, called ``endotacticity'' \cite{CNP13} (see Section~\ref{sect:endotacticity} for its definition); they proposed a \emph{permanence conjecture}\footnote{The permanence conjecture states that \emph{every endotactic mass-action system is permanent} For the definition of permanence, see Definition~\ref{def:persistence+boundedness+permanence}.}
 for endotactic reaction systems which yields GAC. In this way, the original conjecture was put in a broader context, which makes it possible to   prove the conjecture from an alternative perspective than constructing a Lyapunov function for the global stability of the Birch point \footnote{Note that existence of a  Lyapunov function for \emph{local} stability of the Birch point is known \cite{F79}}.}

{With the invention of this concept, by constructing piecewise linear Lyapunov functions to show the existence of a compact global attractor in the interior of the positive cone $\mathbb{R}^d_+$, it is proved that \emph{two-species endotactic mass-action  systems}\footnote{It is noteworthy that the result is proved for a larger class of systems\---the so-called $\kappa$-variable systems where the reaction rate constants are replaced by positive time-varying functions (the so-called ``tempering'' in \cite{GMS14}) with a compact interval range as a subset of $\mathbb{R}_{++}$.}  \emph{are permanent} \cite{CNP13} in 2013. 
The construction of such a  piecewise linear Lyapunov function seems to rely on a deep understanding of endotacticity of the reaction graph; or the other way round, ``endotacticity'' might be born as a generic requirement to be fulfilled in order for a piecewise linear function to be a Lyapunov function for general (two-species) reaction systems.  

Subsequently, by novelly introducing into reaction networks an adapted concept called ``framed jets'' \cite{MP08} from combinatoric geometry, Gopalkrishnan, Miller and Shiu \cite{GMS14} in 2014 provided an equivalent definition of ``endotacticity'' \cite[Lemma~6.22]{GMS14} and proved the permanence conjecture for a subset of endotactic mass-action systems which are \emph{not} limited to two dimensions\---the so-called ``strongly endotactic'' mass-action reaction systems. A crucial step is to prove that the well-known pseudo-Helmholtz engergy type local Lyapunov function for complex-balanced mas-action systems is one for the existence of a compact global attractor away from the boundary \cite[Theorem~7.5]{GMS14}. Both milestone works \cite{CNP13,GMS14} justify that the endotacticity (or strong endotacticity) of the reaction graph implies the  \emph{dissipativity} of the reaction system, in respective contexts. To the best knowledge of the author, there seems to have been no analogous (and rigorous) systematic results since then. For more results on the topic of the permanence conjecture and the GAC, the interested reader may refer to \cite{H74,DAS07,A08,G09,CDSS09,AS10,SJ11,A11a,A11a,P12,CNP13,GMS14,F19}.} 
\smallskip

\noindent\textbf{Main Results.} This paper establishes the following main results: (I) \textbf{Theorem~\ref{thm:A-endotactic=endotactic}}, (II) \textbf{Theorem~\ref{thm:endotactic=weakly-reversible}},  (III) \textbf{Theorem~\ref{thm:deterministic-stability-endotactic}}, and (IV) \textbf{Theorem~\ref{thm:implication-among-three-RNs}}. In particular, we \emph{confirm the permanence conjecture for first order reaction systems} by establishing a \emph{characterization} of the associated reaction graphs.  Out of independent interests, this  characterization of first order endotactic reaction graphs a starting point tries to pave a way to understand the following general question: 
\smallskip

\noindent \emph{\textbf{How much different is endotacticity from strong endotacticity for reaction graphs? In other words, what is their difference set?}}
\smallskip

{For first order reaction graphs, we show that (I-1) \emph{the difference set consists of reaction graphs as a joint of a strongly endotactic reaction graph and a non-empty weakly reversible reaction graph sharing no common species} (Theorem~\ref{thm:A-endotactic=endotactic}).}

{The understanding of the above question in general contexts (not confined to first order reaction graphs) would help to address the permanence conjecture, based on permanence of  strongly endotactic reaction systems already  established in \cite{GMS14} as well as Theorem~\ref{thm:deterministic-stability-endotactic} in this paper:} To close the permanence conjecture, it suffices to study  \emph{nonlinear endotactic but not strongly endotactic mass-action systems of more than two-species}.

To reveal the possible rich dynamics of endotactic reaction systems (and hence the non-triviality of proving the permanence conjecture), it has been discovered that strongly endotactic mass-action systems can admit infinitely many equilibria \cite{KD23}. Hence it might also be appealing to address the question: 
\smallskip

\noindent\emph{\textbf{What (strongly) endotactic mass-action reaction systems have a unique globally stable equilibrium?}}
\smallskip

Besides, we  prove that (I-2) \emph{endotacticity of a given reaction graph can be verified by  \emph{finite} parallel sweep tests} (Theorem~\ref{thm:A-endotactic=endotactic}). It is noteworthy that such desirable finiteness is only proved for \emph{planar} reaction graphs \cite[Proposition~4.1]{CNP13}, and there seems to lack a rigorous proof in higher (than two) dimensions. 

Apart from the  parallel sweep test, we show that (II) \emph{for first order reaction graphs, endotacticity can also be verified by checking \emph{weak reversibility}} (see Section~\ref{sect:RN} for its definition)
\emph{of an auxiliary first order reaction graph} (Theorem~\ref{thm:endotactic=weakly-reversible}). In other words, \emph{every  first order endotactic reaction graph has a weakly reversible deficiency zero (WRDZ) realization}.  \emph{It remains unknown \textbf{if a weakly reversible realization exists for general  endotactic reaction graphs.}} We leave it for future work.

{As a consequence of the two main results Theorem~\ref{thm:A-endotactic=endotactic} and Theorem~\ref{thm:endotactic=weakly-reversible}, we also prove that \emph{for every  first order endotactic  mass-action system,  \emph{there exists a unique equilibrium   (with an explicit representation) within each stoichiometric compatibility class}}  (Theorem~\ref{thm:exisitence-of-positive-equilibrium}), and (III) \emph{the equilibrium is (positive and) globally asymptotically stable in the (positive) stoichiometric compatibility class} (Theorem~\ref{thm:deterministic-stability-endotactic}),} which yields  the permanence conjecture for first order reaction systems.

Out of independent interest, we also establish (IV) \emph{a necessary and sufficient condition for endotacticity of a reaction graph} (Theorem~\ref{thm:implication-among-three-RNs}), which is \emph{not} limited to first order reaction graphs.

\emph{Therefore the above highlighted questions may jointly make a road map to a better  understanding of endotactic reaction graphs in general and hence contributing to the permanence conjecture.} 
\smallskip

\noindent\textbf{Nontriviality of The Main Results.} Below we provide several simple examples to illustrate the non-triviality of the main results. Despite first order reaction graphs appear simple, it seems not completely trivial to determine if such a given reaction graph is endotactic.

\begin{example}
 Consider the following two reaction graphs
\begin{center}\begin{tikzcd}[row sep=1em, column sep = 1em]
\cG_1\colon 0\rar{}&S_1\rar{}&S_2\arrow[bend right=15]{ll}{}&S_3\lar \arrow[red,bend left=15]{ll}{} & \cG_2\colon  0\rar{}&S_1\rar{} \arrow[red,bend right=15]{rr}{}&S_2\arrow[bend right=15]{ll}{}&S_3\lar 
\end{tikzcd}
\end{center}
Since $\cG_2$ is weakly reversible, it is endotactic \cite{CNP13} (see also Corollary~\ref{cor:simplified-criterion-for-endotacticity}). In contrast, the finite sweep test \cite[Proposition~4.1]{CNP13} does \emph{not} apply to $\cG_1$ since it is \emph{three}-dimensional. Despite these two reaction graphs only differ by one edge (colored in red), it turns out that $\cG_1$ is \emph{not} endotactic by Theorem~\ref{thm:endotactic=weakly-reversible} (with $(\cG_1)^{\spadesuit} = \cG_1$; see Section~\ref{sect:endotacticity-first-order} for the meaning of the notation $(\cG_1)^{\spadesuit}$). 
\end{example}
Although for linear ODEs, the  spectral property of the adjacency matrix is sufficient to determine the unique existence of the global asymptotically stable equilibrium,  it turns out that in the context of first order endotactic reaction systems, the adjacency matrix may \emph{not} be Hurwitz stable. 
\begin{example}\label{ex:Introductory-example}
Consider the following reaction network $\cG$ of species $S_i$ for $i=1,\ldots,5$
\begin{center}\begin{tikzcd}[row sep=1em, column sep = 1em]
\cG\colon S_2\rar{}&S_1\rar{}&0\rar{}&S_1+S_2&S_3\rar{}&S_4\rar{}&S_5\arrow[bend left=15]{ll}{}
\end{tikzcd}
\end{center}
{Note that  $\cG$ is neither weakly reversible nor strongly endotactic;  indeed, $\cG$ is \emph{not} $(0,1,2,2,2)$-strongly endotactic (see Section~\ref{sect:endotacticity} for its definition). 
Straightforward yet tedious calculation shows that regardless of the positive reaction rate constants, the coefficient matrix is singular and that there exists a unique positive equilibrium in each positive stoichiometric compatibility class.   
Indeed, $\cG$ is endotactic (by Theorem~\ref{thm:A-endotactic=endotactic}; see Example~\ref{ex:typical-endotactic} for more details), and
by Theorem~\ref{thm:deterministic-stability-endotactic},
 the positive equilibrium in each stoichiometric compatibility class is globally asymptotically stable in that class. Furthermore, the convergence to the equilibrium is exponentially fast (the interested reader may jump to Example~\ref{ex:rate-convergence} for more details).}
\end{example}

\noindent\textbf{Extensions and Applications.} {It is noteworthy that despite the stability  results were established only for first order mass-action systems in this paper, they can be extended to higher order mass-action systems which are not necessarily endotactic, based on the \emph{theory of asymptotically autonomous differential equations} \cite{T92}. 

\begin{example}
Consider the following bimolecular mass-action system:
\begin{center}\begin{tikzcd}[row sep=1em, column sep = 1em]
\cG\colon S_2\rar{2}&S_1\rar{2}&0\rar{2}&S_1+S_2& S_2+S_3\rar{3}&S_1+S_3&S_3\rar{1}&S_4\rar{1}&S_5\arrow[bend left=15]{ll}{2}
\end{tikzcd}
\end{center}
    It is straightforward to verify that $\mathcal{G}$ is \emph{not} ($(1,0,2,2,2)$-)endotactic. Hence related known results \cite{A11a,A11b,CNP13,GMS14} in the literature do not yield permanence and hence the existence of a globally asymptotically stable non-negative equilibrium of $\cG$ in every stoichiometric compatibility class. Nevertheless, by a result of the asymptotically autonomous differential equations \cite{T92}, one can show that there exists a unique globally asymptotically stable positive equilibrium in each positive stoichiometric compatibility class. For more details, see Example~\ref{ex:nonlinear}.
\end{example}

For the coherence of this paper, we do not discuss the full extension along this line in this paper, but leave it to a future work. 

It is worth mentioning that not only applicable to stability problems in the deterministic regime of reaction networks, the main results (Theorem~\ref{thm:A-endotactic=endotactic} and Theorem~\ref{thm:endotactic=weakly-reversible}) also play a fundamental role in studying  \emph{exponential ergodicity} of first order endotactic \emph{stochastic} mass-action reaction systems in a companion work (see Remark~\ref{re:A-endotactic=endotactic}(iv) for more details).}

\subsection*{Outline.}
We introduce notation and review necessary terminology in  graph theory in the next section. Then we introduce reaction networks in Section~\ref{sect:RN}. In Section~\ref{sect:endotacticity}, we prove some properties of endotactic reaction graphs in general. 
Theorem~\ref{thm:A-endotactic=endotactic} and Theorem~\ref{thm:endotactic=weakly-reversible} for first order endotactic reaction graphs are established  in Section~\ref{sect:endotacticity-first-order}. As an application of the main results established in Section~\ref{sect:endotacticity-first-order}, Theorem~\ref{thm:deterministic-stability-endotactic} is given in Section~\ref{sect:stability}. 
\endgroup

\section{Preliminaries}
\subsection{Notation}

Let $\mathbb{R}$ be the set of real numbers, $\mathbb{R}_+$ the set of non-negative real numbers, and $\mathbb{R}_{++}$ the set of positive real numbers. Let $\mathbb{N}_0$ and $\mathbb{N}$ be the set of non-negative integers and that of positive integers, respectively. For $d\in\mathbb{N}$, let $[d]=\{i\}_{i=1}^d$ and $[d]_0=[d]\cup\{0\}$. For every $x\in\mathbb{R}^d$, let $\|x\|_1\coloneqq\sum_{i=1}^d|x_i|$ be its $\ell_1$-norm. For any set $A$, let $\# A$ denote its cardinality. For two disjoint sets $A_1,\ A_2\subseteq\mathbb{R}^d$, we write $A_1\sqcup A_2$ to denote the union of $A_1$ and $A_2$ with the square shape to emphasize that they are disjoint, and similarly $\sqcup_{i=1}^mA_i$ for the union of pairwise disjoint sets $A_i$ for $i\in[m]$, for some $m\in\mathbb{N}\setminus\{1\}$. Let $\Delta_d\coloneqq\{x\in\mathbb{R}^d_+\colon \|x\|_1=1\}$ be the standard simplex of $\mathbb{R}^d$.  Unless stated otherwise, any vector $v\in\mathbb{R}^d$ is a \emph{row} vector throughout this paper.

For any vector $v\in\mathbb{R}^d$, let $\supp v\coloneqq\{i\in[d]\colon v_i\neq0\}$ be the \emph{support} of $v$, $\supp_+ v\coloneqq\{i\in[d]\colon v_i>0\}$ the  \emph{positive support} of $v$, and $\supp_-v=\supp v\setminus\supp_+$ the \emph{negative support} of $v$.
 For a set $V\subseteq\mathbb{R}^d$, let $\supp V\coloneqq\cup_{v\in V}\supp v$ be the support of $V$.  A vector $v=(v_1,\ldots,v_d)\in\mathbb{R}^d_+$ is called a \emph{non-negative} vector  and denoted $v\ge0$; a   vector $v\in\mathbb{R}^d_{++}$ is called a \emph{positive} vector and denoted $v>0$; and a vector $v\in\mathbb{R}^d$ is \emph{negative} if $-v$ is positive.  Let $v^{\perp}\coloneqq\{u\in \mathbb{R}^d\colon u\cdot v^T=0\}$ be the \emph{orthogonal complement} of $v$.  Let $\{e_i\}_{i=1}^d$ be the standard orthonormal basis of $\mathbb{R}^d$.
 Let $\mathbf{1}_d\coloneqq\sum_{i=1}^de_i$, and we simply write $\mathbf{1}$ whenever the dimension $d$ is clear from the context. In contrast, we are a bit sloppy about the use of $0$ without the bold font, which can stand for either a scalar or a vector depending on the context.  Let $\mathcal{M}_d(\mathbb{R})$ be the set of all $d$ by $d$ matrices with real entries. For $A\in\mathcal{M}_d(\mathbb{R})$, let $A^T$ denote its transpose.

\subsection{Graph Theory}

Let $\cG=(\cV,\cE)$ be a simple directed graph. \emph{Throughout, a simple directed graph is called a graph for short.} A directed graph $\cG$ is empty and denoted
$\emptyset$ if it consists of no vertex (and hence no edge either). The number of edges \emph{to} a vertex in a directed graph is the \emph{in-degree} of that vertex, and the number of edges \emph{from} a vertex in a directed graph is the \emph{out-degree} of that vertex. A vertex having zero in-degree and zero out-degree is an \emph{isolated vertex}.  Recall that a graph $\mathcal{G}'=(\cV',\cE')$ is a \emph{subgraph} of $\cG=(\cV,\cE)$ and denoted $\cG'\subseteq\cG$ (or $\cG\supseteq\cG'$) if $\cV' \subseteq \cV$ and $\cE' \subseteq \cE$. For two vertices $y,z\in\cV$, $y\neq z$,  we say $y$ \emph{connects to} $z$ and denoted by $y\rightharpoonup z$ if there exists a directed path $y=y_1\to y_2\ce{->[]} y_3\cdots \ce{->[]} y_k=z$ with edges $y_i\to y_{i+1}\in\cE$ for $i=1,\ldots,k-1$ for some $k\in\mathbb{N}\setminus\{1\}$. If $y\rightharpoonup z$ and $z \rightharpoonup y$, then we write $y \leftrightharpoonup z$. By convention, $y\leftrightharpoonup y$ for every $y\in\cV$. \emph{A spanning tree} of a directed graph is an \emph{acyclic} subgraph  of $\cG$ sharing the same full set of vertices, and with one  vertex, called the \emph{root}, that connects to all other vertices.

Hence $\rightharpoonup$ induces a partial order on $\cV$, and $\leftrightharpoonup$ induces an equivalence relation on $\cV$; moreover,  any equivalent class defined by $\leftrightharpoonup$ is a strongly connected component of $\cG$. Let  $\cG_1=(\cV_1,\cE_1)$ and $\cG_2=(\cV_2,\cE_2)$ be  two strongly connected components of $\cG$. If there exists $y\in\cV_1$ and $y'\in\cV_2$ such that $y\rightharpoonup y'$
(and hence $y'\not\rightharpoonup y$), then we denote by $\cG_1\prec \cG_2$ and say $\cG_1$ is $\prec$-\emph{smaller} than $\cG_2$. Note that $\prec$ also induces a partial order among all strongly connected components of $\cG$. By Zorn's lemma, for every directed graph, there always exists a $\prec$-maximal strongly connected component as well as a $\prec$-minimal strongly connected component.

\section{Reaction networks}\label{sect:RN}
We will then move on to introduce terminologies of reaction networks. We mainly follow the convention of  CRNT \cite{F19}; slight \emph{discrepancy} in term or notation  without causing unnecessary confusion may be expected in this paper for the ease of exposition.

A \emph{reaction graph} (of $d$ species) is an  \emph{unweighted} (possibly empty) simple directed graph $\cG=(\cV,\cE)$ embedded in $\mathbb{R}^d$ \emph{without any isolated vertex}. A non-empty reaction graph is also known as a \emph{Euclidean embedded graph} \cite{C19}.

Every unit vector $e_i$ for $i\in[d]$ in $\mathbb{R}^d$ is called a \emph{species}, and alternatively denoted by the symbol $S_i$. Given a \emph{non-empty} reaction graph $\cG=(\cV,\cE)$,  every vertex in $\cV$ is called a \emph{complex}.  Every directed edge $y\ce{->[]} y'\in\cE$ from a complex $y$ to a complex $y'$ is a bona fide vector in $\mathbb{R}^d\setminus\{0\}$, called a \emph{reaction};
$y'-y$ is called the \emph{reaction vector},
 $y$ is called the \emph{source} of the reaction and $y'$ the \emph{target}.
As we will see, the set of sources of all reactions, in contrast to that of targets of all reactions, will appear frequently in this paper and hence deserves a separate notation, $\cV_+$.  \emph{Throughout this paper, $\cV_+$  will automatically associate with the reaction graph $\cG=(\cV,\cE)$.}
 Hence every vertex of a positive out-degree is a \emph{source} and every vertex of a positive in-degree is a \emph{target}. To sum up, every complex is a linear combination of $S_i$, and the set of reactions defines a \emph{relation} on the set of complexes. The triple set of species, complexes and reactions is called a \emph{reaction network}. For instance in Example~\ref{ex:Introductory-example} in the Introduction,  $S_i$ for $i=1,\ldots,5$ are species, $\cV_+=\{0,e_1,e_2,e_3,e_4,e_5\}$
 consists of sources,
 and there are in total six reactions.

A reaction graph is called \emph{weakly reversible} if there exist no two strongly connected components that are weakly connected. Hence \emph{an empty graph is weakly reversible.} The linear span of reaction vectors of a reaction graph $\cG$ in the real field $\mathbb{R}$ is called the \emph{stoichiometric subspace} of $\cG$:
$$\mathsf{S}_{\cG}\coloneqq \text{span}\{y'-y\colon y\ce{->[]} y'\in\cE\}$$ By convention, $\mathsf{S}_{\emptyset}=\{0\}$ and $\mathsf{S}_{\emptyset}^{\perp}=\mathbb{R}^d$. If $\mathsf{S}_{\cG}\cap\mathbb{R}^d_{++}\neq\emptyset$, then any vector in $\mathsf{S}_{\cG}\cap\mathbb{R}^d_{++}$ is a \emph{conservation law} of $\cG$. The dimension of $\mathsf{S}_{\cG}$ is referred to as the \emph{dimension} of $\cG$. Each  translation of the stoichiometric subspace by a point in $\mathbb{R}^d$ confined to $\mathbb{R}^d_+$ is a \emph{stoichiometric compatibility class} \cite{F87}; in particular, the interior of a stoichiometric compatibility class whenever it is non-empty, is a \emph{positive stoichiometric compatibility class} \cite{F87}. Hence the (Hausdorff) dimension of any \emph{positive}  stoichiometric compatibility class equals that of the stoichiometric subspace.

Let $\cG=(\cV,\cE)$ be a reaction graph. Let $\ell_{\cG}$ be the number of  strongly connected components \emph{containing at least two vertices} of $\cG$.  The \emph{deficiency} of $\cG$ is defined to be the integer $\#\cV-\dim\mathsf{S}_{\cG}-\ell_{\cG}$, which is always non-negative  \cite{F19}. Hence \emph{an empty reaction graph is of deficiency zero}. Speaking of a weakly reversible  reaction network, its deficiency is the number of \emph{independent} equations for the edge weights of the graph in order for the reaction network to be \emph{complex-balanced} \cite{J12,F19} (for the definition of complex-balanced reaction network, see \eqref{eq:complex-balancing} in Section~\ref{sect:stability}.).  For each reaction $y\ce{->[]} y'\in\mathcal{E}$, the $\ell_1$-norm $\|y\|_1$ of the source is called the \emph{order of the reaction}\footnote{Indeed, the $\ell_1$-norm of the source is termed as \emph{molecularity} of the reaction in the literature of chemistry \cite{M67}. With mass-action kinetics, it agrees with the \emph{order} of the reaction, which is the sum of exponents of the monomial propensity function of the reaction. Since the paper focuses on dynamics of mass-action systems, we abuse ``order'' for ``molecularity'' to save the latter term.}, and $\max\{\|y\|_1,\|y'\|_1\}$ is called the \emph{gross order of the reaction}.
Obviously for a given reaction, its gross order is no smaller  than its order.
  Let $r=\max_{y\in\cV_+}\|y\|_1$ be the \emph{order of $\mathcal{G}$} and
$r'=\max_{y\in\cV}\|y\|_1$ the \emph{gross order of $\cG$}.  Analogously, $r'\ge r$. It is noteworthy that since the reaction graph is embedded in $\mathbb{R}^d$, the order of a reaction graph, despite is always non-negative, may \emph{not} be an integer.
In particular, a reaction graph embedded in $\mathbb{N}^d_0$ of gross order one is called \emph{monomolecular}. Let $\cG^*=(\cV^*,\cE^*)$ with $\cE^*=\{y\ce{->} y'\in\mathcal{E}\colon \|y\|_1=\max_{z\in\cV_+}\|z\|_1\}$ and $\cV^*=\{y,y'\colon y\ce{->}y'\in \cE^*\}$ be the sub reaction graph of $\cG$  consisting of purely  highest order
  reactions of $\cG$. A reaction graph $\cG$ is \emph{homogeneous} if $\cG=\cG^*$.  Throughout, unless stated otherwise, \emph{all reaction graphs are assumed to have the same set of species $\mathcal{S}=\{S_i\}_{i=1}^d$ for some $d\in\mathbb{N}$, particularly when they appear in a context for comparison}.

A species $S_i$ is \emph{redundant} if $e_i\in\mathsf{S}^{\perp}_{\cG}$, i.e., $(y'-y)_i=0$ for all reactions $y\ce{->[]} y'$ of $\mathcal{G}$; in other words, there is no molecule change for species $S_i$ in any reaction. For the ease of exposition and without loss of generality (w.l.o.g.), we assume throughout that \emph{reaction graphs have no redundant species}, namely $\{e_i\}_{i\in[d]}\cap\mathsf{S}^{\perp}_{\cG}=\emptyset$. To rephrase this running assumption,  we exclude the case where all reactions lie on a finite set of hyperplanes of a positive dimension whose reaction vectors are perpendicular to $e_i$ for any $i\in[d]$.  Otherwise, in the study of reaction networks, one can always embed the \emph{kinetic} effect induced by redundant species in  the reaction rate constants, and decompose the reaction network into finitely many sub reaction networks and study their (dynamical) properties separately.

\section{Endotactic reaction graphs}\label{sect:endotacticity}
In this section, 
we first establish some properties of  \emph{endotactic reaction graphs}, which themselves are \emph{not} limited to first order reaction graphs and hence are of independent interest.

A wide class of reaction networks are \emph{endotactic reaction networks} (i.e., endotactic reaction graphs in this paper), which were introduced in \cite{CNP13}. A subset of endotactic reaction networks with some additional properties (see below for the precise definition) are called \emph{strongly endotactic reaction networks} \cite{GMS14}. Both insightful concepts were introduced to investigate \emph{permanence} and \emph{persistence} of reaction systems \cite{CNP13,P12,GMS14}.

Let $\cG=(\cV,\cE)$ be a reaction graph embedded in $\mathbb{R}^d$. For every
$u\in\mathbb{R}^d$,  define a sub reaction graph $\cG_u=(\cV_u,\cE_u)$, where $$\mathcal{E}_u\coloneqq\{y\ce{->[]} y'\in\mathcal{E}\colon y'-y\notin u^{\perp}\},\quad \cV_u\coloneqq\{y, y'\colon y\ce{->[]} y'\in\mathcal{E}_u\}$$
 In other words, the possibly empty reaction graph $\cG_u$ \emph{consists of all reactions in $\cE$ whose reaction vectors are not perpendicular to $u$}. Let $\cV_{u,+}\subseteq\cV_u$ be the set of sources of $\cG_u$. Note that the vector $u$ induces a \emph{total} order on $\cV_{u}$: $$y>_u z \iff (y-z)\cdot u^T > 0;\ y=_uz \iff (y-z)\cdot u^T = 0;\ y<_u z \iff (y-z)\cdot u^T < 0$$ Two complexes $y,z\in\cV$ are called \emph{$u$-equal} if $y=_uz$. Essentially, all $u$-equal points in $\mathbb{R}^d$ lie in a subspace as a translation of $u^{\perp}$. Hence  $\cG_u$ is obtained by \emph{removing all edges between two $u$-equal vertices, as well as all resulted isolated vertices to obtain a subgraph of $\cG$}. Moreover, a complex is \emph{$u$-maximal} in a subset (e.g., $\cV_{u,+}$) of $\cV$ if under this total order it is maximal in that set.  Let $\supp_u\cG$ be the set of $u$-maximal elements in $\cV_{u,+}$. Note that \emph{all elements in $\supp_u\cG$ are $u$-equal}.

\begin{definition}\label{def:endotacticity}
Let $\cG$ be a  reaction graph. Any reaction $y\ce{->[]} y'\in\cG$ satisfying $y<_uy'$ and
 $y\in\supp_u\cG$  for some $u\in\mathbb{R}^d$ is called a \emph{$u$-endotacticity violating reaction} of $\cG$, or simply called \emph{endotacticity violating reaction} of $\cG$ when $u$ is deemphasized.
We say the reaction graph $\mathcal{G}$ is \emph{$u$-endotactic} if $\cG$ has no $u$-endotacticity violating reaction.  Furthermore, a $u$-endotactic reaction graph $\cG$ is \emph{$u$-strongly endotactic} if additionally
$\supp_u\cG$ contains a $u$-maximal element in $\cV_+$. Given any subset $\cU\subseteq\mathbb{R}^d$,  we say $\cG$ is \emph{$\mathcal{U}$-endotactic} (\emph{$\mathcal{U}$-strongly endotactic}, respectively) if $\cG$ is $u$-endotactic ($u$-strongly endotactic, respectively) for every  $u\in\cU$. \emph{By convention, no reaction graph but the empty reaction graph is $\emptyset$-endotactic.}
 In particular, $\cG$ is \emph{lower-endotactic} (\emph{lower-strongly endotactic}, respectively) if it is $\mathbb{R}^d_+$-endotactic ($\mathbb{R}^d_+$-strongly endotactic, respectively), and $\cG$ is   \emph{endotactic} (\emph{strongly endotactic}, respectively) if $\cG$ is $\mathbb{R}^d$-endotactic ($\mathbb{R}^d$-strongly endotactic, respectively). In other words, \emph{$\cG$ is endotactic if and only if $\cG$ has no endotacticity-violating reaction}. Hence an empty graph is endotactic but \emph{not} strongly endotactic.
\end{definition}

By definition, \emph{$\cG$ is endotactic if and only if it is $\mathbb{R}^d\setminus\mathsf{S}^{\perp}_{\cG}$-endotactic.}

To determine whether a given reaction graph is $u$-endotactic for a vector $u\in\mathbb{R}^d\setminus\mathsf{S}^{\perp}_{\cG}$, one can move a hyperplane parallel to $u^{\perp}$ towards the direction of $u$ as a bona fide vector to \emph{sweep} all reactions  also as bona fide vectors in $\mathbb{R}^d\setminus u^{\perp}$; if the hyperplane will \emph{first sweep the target but not the source} of the first reaction, then the reaction \emph{passes the test}  and is verified to be $u$-endotactic. This is the so-called  \emph{parallel sweep test} \cite{CNP13}.

For two dimensional reaction graphs, endotacticity is equivalent to $\mathcal{U}$-endotacticity for a finite set $\mathcal{U}$  \cite[Proposition~4.1]{CNP13}, which depends on the reaction graph. To the best knowledge of the author, it seems open if analogous results hold for higher dimensions.

In the following, we will provide a necessary and sufficient condition (Theorem~\ref{thm:implication-among-three-RNs}) for a reaction graph to be endotactic.

Given reaction graph $\cG=(\cV,\cE)$, define the following sub reaction graph of $\cG$:
\begin{equation}\label{eq:reaction-tree}
\cG^{\RS{lir}}=(\cV^{\RS{lir}},\cE^{\RS{lir}}),
\end{equation} where $\cE^{\RS{lir}}\subseteq\cE$ consists of reactions \emph{whose source and target are in different strongly connected components} of $\cG$, and $\cV^{\RS{lir}}$ is the set of complexes of reactions in $\cE^{\RS{lir}}$. Note that $\cG^{\RS{lir}}$ is a (possibly empty)  sub reaction graph of $\cG$, and is a \emph{tree} whenever it is non-empty.

To prove Theorem~\ref{thm:implication-among-three-RNs}, we rely on the lemma below, which presents a necessary condition for a reaction to be endotacticity-violating.
\begin{lemma}\label{le:endotacticity-violating}
Let $\cG$ be a reaction graph embedded in $\mathbb{R}^d$. Let $w\in\mathbb{R}^d$. Assume $y\ce{->[]} y'\in\cE$ is a $w$-endotacticity-violating reaction. Then for all $z\in\cV$ such that $y'\rightharpoonup z$, we have $y'=_wz$  and $z\notin\cV_{w,+}$. In particular, for any endotacticity-violating reaction, its source and its target must lie in different strongly connected components of $\cG$.
\end{lemma}
\begin{proof}
Assume the former conclusion is true. Let $y\ce{->[]} y'$ be an endotacticity-violating reaciton. By contraposition, $y'\not\rightharpoonup y$, which yields that $y$ and $y'$ are in different strongly connected components. Hence it suffices to prove the former conclusion. Since $y\ce{->[]} y'$ is $w$-endotacticity violating, we have 
$y\in\supp_w\cG$ and $y'\notin\cV_{w,+}$ (otherwise it would contradict the $w$-maximality of $y$ in $\cV_{w,+}$). This implies that for any $z\in\cV$ such that $y'\ce{->[]} z\in\cE$, we have $y'\ce{->[]} z\notin\cE_w$, i.e.,  $z=_wy'$. 
 By induction, the desired (former) conclusion follows.
\end{proof}
\begin{corollary}\label{cor:simplified-endotacticity-def}
Let $\cG$ be a reaction graph embedded in $\mathbb{R}^d$. Then $\cG$ is endotactic if and only if $\cG$ is $\mathbb{R}^d\setminus\mathsf{S}_{\cG^{\RS{lir}}}^{\perp}$-endotactic.
\end{corollary}
\begin{proof}
\noindent $\implies$ It is obvious since $\mathbb{R}^d\setminus\mathsf{S}_{\cG^{\RS{lir}}}^{\perp}\subseteq\mathbb{R}^d$.
\medskip

\noindent $\impliedby$ We prove it by contraposition. Suppose $\cG$ is \emph{not} endotactic, then there exists a $w$-endotacticity violating reaction $y\ce{->[]} y'\in\cE$ of $\cG$. By Lemma~\ref{le:endotacticity-violating}, $y\ce{->[]} y'\in \cE^{\RS{lir}}$. Note that $w\notin(y'-y)^{\perp}$, which implies that $w\in\mathbb{R}^d\setminus\mathsf{S}_{\cG^{\RS{lir}}}^{\perp}$. This  contradicts that $\cG$ is $\mathbb{R}^d\setminus\mathsf{S}_{\cG^{\RS{lir}}}^{\perp}$-endotactic.
\end{proof}
\begin{theorem}\label{thm:implication-among-three-RNs}
Given a reaction graph $\cG$ embedded in $\mathbb{R}^d$, let $\cG^{\RS{lir}}$ be defined in \eqref{eq:reaction-tree}. Then
$\cG$ is endotactic if and only if there exists a (possibly empty) endotactic reaction graph $\widehat{\cG}$  
such that $\cG^{\RS{lir}} \subseteq \widehat{\cG}\subseteq \cG$.
\end{theorem}
\begin{proof}
The forward implication (``only if'' part) is trivial since one can simply take $\widehat{\cG}=\cG$. It remains to show the reverse implication. We prove it by contraposition. Suppose $\cG$ is \emph{not} endotactic. Then there exists a $w$-endotacticity violating reaction  $y\ce{->[]} y'\in\cE$  of $\cG$ with a 
$y\in\supp_w\cG$  for some $w\in\mathbb{R}^d$.  By Lemma~\ref{le:endotacticity-violating}, $y\ce{->[]} y'\in \cE^{\RS{lir}} \subseteq\widehat{\cE}$. Moreover,  
$y\in\supp_w\widehat{\cG}$  since $\widehat{\cV}_{w,+}\subseteq\cV_{w,+}$.  Thus $y\ce{->[]} y'$ is also a $w$-endotacticity violating reaction of $\widehat{\cG}$, contradicting that $\widehat{\cG}$ is endotactic.
\end{proof}

From Theorem~\ref{thm:implication-among-three-RNs} we can recover the following known result \cite{CNP13}.
\begin{corollary}\label{cor:simplified-criterion-for-endotacticity}
Every weakly reversible reaction graph is endotactic.
\end{corollary}
\begin{proof}
Note that $\cG$ is weakly reversible if and only if $\cG^{\RS{lir}}$ is empty, which is endotactic. By Theorem~\ref{thm:implication-among-three-RNs}, $\cG$ is endotactic.
\end{proof}

We also obtain a simple sufficient condition for endotacticity of a reaction graph $\cG$, based on a  property of $\cG^{\RS{lir}}$.
\begin{corollary}\label{cor:simplified-sufficient-condition-for-endotacticity}
Given a reaction graph $\cG$ embedded in $\mathbb{R}^d$, let $\cG^{\RS{lir}}$ be defined in \eqref{eq:reaction-tree}. Assume for every reaction $y\ce{->}y'\in \cE^{\RS{lir}}$, there exists $y'\ce{->}y''\in\cE$ such that $(y'-y) \parallelsum (y''-y')$. Then $\cG$ is endotactic.
\end{corollary}
\begin{proof}
It suffices to show there exists no $w$-endotacticity violating reaction for any $w\in\mathbb{R}^d\setminus\mathsf{S}_{\cG^{\RS{lir}}}^{\perp}$.
We prove it by contradiction. Suppose $y\ce{->}y'$ is a $w$-endotacticity violating reaction for some $w\in\mathbb{R}^d$. By Lemma~\ref{le:endotacticity-violating}, $y'=_wy''$; moreover, $y=_wy'$, which contradicts that $y\ce{->}y'$ is a $w$-endotacticity violating reaction.
\end{proof}

\begin{example}
Consider the following reaction graph
\begin{center}\begin{tikzcd}[row sep=1em, column sep = 1em]
\cG\colon 0\rar[red]{}&S_1\rar{}&2S_1\rar{}&3S_1\arrow[bend left=15]{ll}{}
\end{tikzcd}
\end{center}
Here edges connecting different components are colored in red (in this case, there is a unique one). Note that $$\cG^{\RS{lir}}\colon 0\ce{->}S_1,$$ which is  \emph{not} endotactic. However, since this is a one-dimensional reaction graph, any other reaction as a bona fide vector is parallel to $0\ce{->}S_1$. By  Corollary~\ref{cor:simplified-sufficient-condition-for-endotacticity}, $\cG$ is endotactic.
\end{example}
Not surprisingly, the assumption in Corollary~\ref{cor:simplified-sufficient-condition-for-endotacticity} is \emph{not} necessary for endotacticity.
\begin{example}
Revisit Example~\ref{ex:Introductory-example}:
\begin{center}
\begin{tikzcd}[row sep=1em, column sep = 1em]
\cG\colon S_2\rar[red]{}&S_1\rar[red]{}&0\rar[red]{}&S_1+S_2&S_3\rar{}&S_4\rar{}&S_5\arrow[bend left=15]{ll}{}
\end{tikzcd}
\end{center}
Note that $\cG$ is endotactic with
\[\cG^{\RS{lir}}\colon S_2\ce{->}S_1\ce{->}0\ce{->}S_1+S_2\] Nevertheless, $S_1\ce{->}0\in\widetilde{\cG}$ while there exists no other reaction in $\cG$ as a bona fide vector parallel to this reaction.
\end{example}
It seems natural to speculate if an analogue of Theorem~\ref{thm:implication-among-three-RNs} exists for strong endotacticity; in other words, if \emph{the existence of a reaction graph $\widehat{\cG}\neq\cG$ such that $\cG^{\RS{lir}} \subseteq\widehat{\cG} \subseteq \cG$ 
implies strong endotacticity of $\cG$}.  Indeed, such an analogue fails to be true.

\begin{example}\label{ex:strong-endotacticity-fails-to-be-preserved}
Consider the following reaction graph
\[\cG\colon 0\textcolor{red}{\ce{->}}S_1\textcolor{red}{\ce{<-}}2S_1\quad 3S_1\ce{<=>}3S_1+S_2\]
It is readily confirmed that
\[\cG^{\RS{lir}}\colon 0\ce{->}S_1\ce{<-}2S_1\]
is strongly endotactic, since as a one-dimensional reaction graph embedded in the real line, 
its ``leftmost'' source points right and its ``rightmost'' source points left \cite[Remark~3.11]{GMS14}. Hence by Theorem~\ref{thm:implication-among-three-RNs}, $\cG$ is endotactic. Nevertheless, $\cG$ is \emph{not} strongly endotactic since, e.g., it is not $(1,0)$-strongly endotactic.
\end{example}

In general that endotacticity of $\cG$ cannot imply endotacticity of its \emph{proper} subgraphs $\widehat{\cG}$ as super
graphs of $\cG^{\RS{lir}}$.
\begin{example}\label{ex:G-star-counterexample}
Consider $$\cG\colon 0\textcolor{red}{\ce{->[]}} S_1\ce{<=>[][]}2S_1$$ It is justifiable by the same argument as in Example~\ref{ex:strong-endotacticity-fails-to-be-preserved} that
 $\cG$ is endotactic with $$\cG^{\RS{lir}}\colon  0\ce{->[]} S_1$$
Now consider the other two proper subgraphs of $\cG$ as super 
graphs of $\cG^{\RS{lir}}$: $$\widehat{\cG}_1\colon 0\ce{->[]} S_1\ce{->[]}2S_1\quad \widehat{\cG}_2\colon 0\ce{->[]} S_1\ce{<-[]}2S_1$$ It is easy to verify that among the three proper sub reaction graphs of $\cG$, neither $\cG^{\RS{lir}}$ nor $\widehat{\cG}_1$ is endotactic while $\widehat{\cG}_2$ is endotactic.
\end{example}

Next, we demonstrate the applicability of Theorem~\ref{thm:implication-among-three-RNs}.

\begin{example}
Consider the following reaction graph
\begin{center}\begin{tikzcd}[row sep=1em, column sep = 1em]
\cG\colon S_3+S_4\arrow[bend right=10]{rr}& 2S_3\lar& 2S_2\rar[red]\lar&S_1+S_2\arrow[bend left=0]{ld}&2S_1\arrow[red]{l}&\hspace{-.7cm}\ce{<=>}S_2+S_3\\
&&S_3\arrow[]{rr}&&S_4\arrow[bend left=0]{lu}
\end{tikzcd}
\end{center}

Note that $\cG$ is of dimension $4$ and hence \cite[Proposition~4.1]{CNP13} does not apply. Nevertheless, it is easy to observe that $$\cG^{\RS{lir}}\colon 2S_2\ce{->[]} S_1+S_2\ce{<-[]} 2S_1$$
as a one-dimensional reaction graph, is endotactic  \cite[Remark~3.11]{GMS14}. By Theorem~\ref{thm:implication-among-three-RNs}, $\cG$ is also endotactic.
\end{example}
\begin{example}
Consider 
\begin{center}\begin{tikzcd}[row sep=1em, column sep = 1em]
\cG\colon 0\arrow[red]{r}&S_1\rar&2S_1\rar&S_1+S_2\rar&2S_3\arrow[bend left=15]{lll}
\end{tikzcd}
\end{center}
Note that $\cG$ is of dimension $3$. Moreover, it is readily verified that  $$\cG^{\RS{lir}}\colon 0\ce{->}S_1$$ is \emph{not}  endotactic. However, $0\ce{->}S_1$, as the unique reaction of $\cG^{\RS{lir}}$, is parallel to $S_1\ce{->}2S_1\in\cE$ as bona fide vectors.
 By Corollary~\ref{cor:simplified-sufficient-condition-for-endotacticity}, $\cG$ is endotactic.
\end{example}

Essentially, the reason why $\cG^{\RS{lir}}$ is not endotactic while $\cG$ may be endotactic is that $\cE\setminus \cE^{\RS{lir}}$ may contain reactions with $w$-maximal sources in $\cV_+$ so that \emph{a $w$-endotacticity violating reaction $y\to y'\in \cE^{\RS{lir}}$ of $\cG^{\RS{lir}}$ may not be a $w$-endotacticity violating reaction of $\cG$}.

Next, we show that endotacticity is preserved under the \emph{joint}  operation. 

\begin{definition}\label{def:joint-of-RN}
Let $\cG_i=(\cV_i,\cE_i)$ for $i=1,2$ be two  reaction graphs, both embedded in $\mathbb{R}^{d}$. We define the \emph{joint}  of two reaction graphs by the following reaction graph:
\[\mathcal{G}_1\cup\mathcal{G}_2=(\cV_1\cup\cV_2,\cE_1\cup\cE_2)\]
\end{definition}
If two reaction graphs are embedded in different Euclidean spaces $\mathbb{R}^{d_i}$ for $i=1,2$, then one can first \emph{lift} both reaction graphs to reaction graphs embedded in $\mathbb{R}^{d}$, where $d=\max\{d_1,d_2\}$, and define their joint as the joint of their lifted reaction graphs.

 \begin{definition}
 Two reaction graphs $\cG_1$ and $\cG_2$ are called \emph{disjoint} if both of their vertex sets and their edge sets are disjoint: $$\cV_1\cap\cV_2=\emptyset,\quad\cE_1\cap\cE_2=\emptyset$$ In this case, we write the joint of the two reaction graphs as $\cG_1\sqcup\cG_2$.
 \end{definition}

\begin{lemma}\label{le:endotacticity-preserved-under-joint}
Let $\cG_i$ be two reaction graphs  embedded in $\mathbb{R}^d$, for $i=1,2$. If both $\cG_1$ and $\cG_2$ are endotactic, then so is their joint $\cG_1\cup\cG_2$.
\end{lemma}
\begin{proof}
Assume w.l.o.g. that neither $\cG_1$ nor $\cG_2$ is empty. It suffices to verify $w$-endotacticity of $\cG_1\cup\cG_2$ for every $w\in\mathbb{R}^d\setminus\mathsf{S}_{\cG_1\cup\cG_2}^{\perp}$. Let $\cG_{i,w}$ be short for $(\cG_i)_{w}$ for $i=1,2$, and \emph{the same abbreviation rule applies to other sets henceforth}.  Note that $(\cG_1\cup\cG_2)_w=\cG_{1,w}\cup\cG_{2,w}\neq\emptyset$. Let $\cV_{i,w,+}$ denote the set of sources of $\cG_{i,w}$ for $i=1,2$. Let $y$ be any $w$-maximal source in $\cV_{1,w,+}\cup\cV_{2,w,+}$. If $y\in\cV_{1,w,+}$, then $y$ is  $w$-maximal in $\cV_{1,w,+}$. By $w$-endotacticity of $\cG_1$, we have $y>_{w}y'$ for all $y\ce{->[]} y'\in\cE_{1,w}$. Analogously, if $y\in\cV_{2,w,+}$, then $y>_{w}y'$ for all $y\ce{->[]} y'\in\cE_{2,w}$. In sum, $y>_{w}y'$ for all $y\ce{->[]} y'\in\cE_{w}=\cE_{1,w}\cup\cE_{2,w}$. This shows $w$-endotacticity of $\cG_1\cup\cG_2$.
\end{proof}
\begin{remark}
Despite $\cG$ has no redundant species, the two sub reaction graphs $\cG_1$ and $\cG_2$ in Lemma~\ref{le:endotacticity-preserved-under-joint} are indeed \emph{allowed to have redundant species}. For instance, consider $\cG=\cG_1\cup\cG_2$ with $\cE=\cE_1\sqcup\cE_2$ and $$\cG_1\colon S_2\ce{->[]} S_1+S_2\ce{<-[]} 2S_1+S_2;\quad \cG_2\colon S_1\to S_1+S_2\ce{<-[]} S_1+2S_2$$ Note that both sub reaction graphs have redundant species and are endotactic as one-dimensional reaction graphs \cite[Remark~3.11]{GMS14}. Hence by Lemma~\ref{le:endotacticity-preserved-under-joint}, $\cG$ is endotactic.
\end{remark}

Endotacticity may also be preserved under \emph{subtraction}.
\begin{lemma}\label{le:endotacticity-preserved-under-subtraction}
Let $\cG$ be a reaction graph  embedded in $\mathbb{R}^d$. Assume $\cG=\cG_1\sqcup\cG_2$ can be  decomposed into two sub reaction graphs $\cG_1$ and $\cG_2$ of disjoint sets of species. Then $\cG$ is endotactic if and only if $\cG_1$ and $\cG_2$ are both endotactic.
\end{lemma}
\begin{proof}
Assume w.l.o.g. that $\cG_1$ and $\cG_2$ are both non-empty. By Lemma~\ref{le:endotacticity-preserved-under-joint}, it suffices to prove the ``only if'' part. Assume w.l.o.g. that $\cG=\cG_1\sqcup\cG_2$ is endotactic, where $\cG_i=(\cV_i,\cE_i)\neq\emptyset$ for $i=1,2$. W.l.o.g., it suffices to show that $\cG_1$ is endotactic.  Let $\supp \cV_1=\cI_1\subsetneq[d]$ and $\#\cI_1=d_1<d$. Note that $\cG_1$ is a reaction graph of $d_1$ species. 
Next, we will pair each reaction in $\cE_1$ as a bona fide vector in $\mathbb{R}^{d_1}$ with a reaction in $\cE$ as a bona fide vector in $\mathbb{R}^d$. For any $u\in\mathbb{R}^{d_1}\setminus\mathsf{S}^{\perp}_{\cG_1}$, let $w\in\mathbb{R}^d$ be defined as:
 \[w_j=u_j\mathbbm{1}_{\mathcal{I}_1}(j),\quad j\in[d]\] It is easy to observe that $w\in\mathbb{R}^d\setminus\mathsf{S}^{\perp}_{\cG}$. 
 For any $\breve{y}\in\cV_{1,u}$, let $$y_j=\breve{y}_j\mathbbm{1}_{\mathcal{I}_1}(j),\quad j\in[d]$$
Note that $\breve{y}\ce{->}\breve{y}'\in\cE_1$ implies that $y\ce{->}y'\in\cE$; moreover, $$(y-y')\cdot w^T=(\breve{y}-\breve{y}')\cdot u^T\neq0,\quad \forall \breve{y}\ce{->}\breve{y}'\in\cE_{1,u},$$ yielding that $y\ce{->[]} y'\in\cG_w$.  On the other hand, since $\cG_1$ and $\cG_2$ have disjoint sets of species, by the definition of $w$, we have $\supp\cV_2\cap\supp w=\emptyset$ and  $w\in\mathsf{S}_{\cG_2}^{\perp}$. Hence $$\cE_w=\{y\to y'\colon \breve{y}\ce{->[]} \breve{y}'\in\cE_{1,u}\},$$ and 
$y\in\supp_w\cG$ if and only if $\breve{y}\in\supp_u\cG_1$. Since $\cG$ is $w$-endotactic,  for every 
$\breve{y}\in\supp_u\cG_1$ and $\breve{y}\ce{->}\breve{y}'\in\cE_{1,u}$, we have$$(\breve{y}-\breve{y}')\cdot u^T=(y-y')\cdot w^T>0$$ This shows that $\cG_1$ is $u$-endotactic.
\end{proof}

While Lemma~\ref{le:endotacticity-preserved-under-joint} allows the two sub reaction graphs to share complexes and hence species, \emph{disjointness of sets of species}
of the two sub reaction graphs is a crucial assumption for Lemma~\ref{le:endotacticity-preserved-under-subtraction}.
\begin{example}
Consider the following one-species reaction graph $$\cG\colon 0\ce{->[]} S_1\quad 2S_1\ce{<=>[]}3S_1$$
and a decomposition of 
 $\cG$ $$\cG_1\colon 0\ce{->[]} S_1;\quad \cG_2\colon2S_1\ce{<=>[]}3S_1$$ Despite $\cG$ and $\cG_2$ are endotactic, $\cG_1$ is \emph{not}.
\end{example}

\section{Endotacticity of first order reaction graphs}\label{sect:endotacticity-first-order}

Despite the appealing property of endotacticity \cite{P12,CNP13,GMS14} which in certain cases is proved to be sufficient for permanence of \emph{$\kappa$-variable} mass-action systems (where ``$\kappa$-variable'' means edge weights of the reaction graph are allowed to vary in time while remain uniformly bounded away from zero), it remains open in general if endotacticity can be determined by applying finitely many parallel sweep tests 
(i.e., there exists a finite set $\mathcal{U}\subseteq\mathbb{R}^d$ such that $\cG$ is endotactic if it is $\cU$-endotactic). To the best knowledge of the author, such finiteness of parallel sweep tests seems to have been verified only for  reaction graphs embedded in $\mathbb{R}^d$ for $d=1,2$ \cite[Proposition~4.1]{CNP13}. 
In this section, we will show that \emph{endotacticity can be determined by finitely many parallel sweep tests for first order reaction graphs}  (Theorem~\ref{thm:A-endotactic=endotactic}). 
Before presenting Theorem~\ref{thm:A-endotactic=endotactic}, we provide an archetypal example of a first order endotactic reaction graph.
\begin{example}\label{ex:typical-endotactic}
Revisit Example~\ref{ex:Introductory-example}:
\begin{center}\begin{tikzcd}[row sep=1em, column sep = 1em]
\cG\colon S_2\rar{}&S_1\rar{}&0\rar{}&S_1+S_2&S_3\rar{}&S_4\rar{}&S_5\arrow[bend left=15]{ll}
\end{tikzcd}
\end{center}
Let $\cA$ be defined as in \eqref{eq:def-A} below with $d=5$. It is straightforward to verify that $\cG$ is $\cA$-endotactic. Indeed, $\cG$ consists of a weakly connected sub reaction graph containing the zero complex \begin{center}\begin{tikzcd}[row sep=1em, column sep = 1em]
\cG^0\colon S_2\rar{}&S_1\rar{}&0\rar{}&S_1+S_2
\end{tikzcd}
\end{center}
and a sub reaction graph $\cG^{\bullet}=\cG\setminus\cG^0$ which is WRDZ
\begin{center}\begin{tikzcd}[row sep=1em, column sep = 1em]
\cG^{\bullet}\colon S_3\rar{}&S_4\rar{}&S_5\arrow[bend left=15]{ll}
\end{tikzcd}
\end{center}
By an analogue of the criterion for endotacticity given in \cite[Proposition~4.1]{CNP13}, $\cG^0$ is \emph{strongly endotactic}. Since strongly endotactic reaction graphs and weakly reversible reaction graphs are both endotactic \cite{CNP13}, by Lemma~\ref{le:endotacticity-preserved-under-joint}, we know $\cG$ is endotactic. 
\end{example}
Although the argument is not applicable to general first order reaction graphs, it turns out a first order endotactic reaction graph is always the joint of a (possibly empty) strongly endotactic reaction graph and a (possibly empty) WRDZ reaction graph.

Let
\begin{equation}\label{eq:def-A}
\mathcal{A}=\{\pm\sum_{i\in\cI}e_i\colon  \emptyset\neq\cI\subseteq[d]\}
\end{equation}
be a finite set of vectors in $\mathbb{R}^d$.  For any first order reaction graph $\cG$,  let $\cG^0=(\cV^0,\cE^0)$ be the (possibly empty) weakly connected component of $\cG$ containing the zero complex.  Note that $\cG^0\neq\emptyset$ if and only if $0\in\cV$. Let $\cG^{\bullet} \coloneqq \cG\setminus \cG^0$ be the (possibly empty) reaction graph consisting of reactions in $\cE^{\bullet}\coloneqq \cE\setminus\cE^0$ and complexes in $\cV^{\bullet}\coloneqq\cV\setminus\cV^0$.

\begin{theorem}\label{thm:A-endotactic=endotactic}
Let $\cG$ be a first order reaction graph embedded in $\mathbb{N}^d_0$. Assume $\cG$ is $\cA$-endotactic. Then $\cG$ is endotactic. More precisely, $\cG^0$ and $\cG^{\bullet}$ are (possibly empty) endotactic subgraphs of $\cG$ of disjoint sets of species, and $\cG^{\bullet}$ is WRDZ 
while $\cG^0$ is strongly endotactic provided it is non-empty.
\end{theorem}
To prove this result, we need to first establish the following three lemmas.

First, under $\{\mathbf{1},-\mathbf{1}\}$-endotacticity,  we are able to characterize $\cG$ when $\cG^0=\emptyset$.

\begin{lemma}\label{le:homogeneity}
Let $\cG=(\cV,\cE)$ be a non-empty first order reaction graph embedded in $\mathbb{N}^d_0$. Assume $\cG$ is
$\{\mathbf{1},-\mathbf{1}\}$-endotactic. Then $$(1)\ 0\notin\cV\quad \Leftrightarrow\quad (2)\ \cG\ \text{has conservation law}\ \mathbf{1}\quad \Leftrightarrow\quad (3)\ \cG\ \text{is homogeneous},$$ in which case $d>1$, and the set of complexes consists of single-copy species.
\end{lemma}
\begin{proof}
We prove the two bi-implications in a cyclic manner.
\medskip

\noindent $(1)\Rightarrow(2)$. Since $\cG$ is a first order reaction graph without 
the zero complex, every reaction of $\cG$ is of order 1,  and $\cG=\cG^{\bullet}$ is homogeneous. Hence all reactants are $\mathbf{1}$-maximal in $\cV_+$. By $\mathbf{1}$-endotacticity of $\cG$, $$(y'-y)\cdot\mathbf{1}^T\le0,\quad \forall y\ce{->[]} y'\in\cE$$ Next we prove that $\cG$ has conservation law $\mathbf{1}$ by contradiction. Suppose $\mathbf{1}\notin\mathsf{S}_{\cG}^{\perp}$. We will prove that $\cG$ has a $-\mathbf{1}$-endotacticity violating reaction in $\cE$ to achieve the contradiction. Note that $\cG_{-\mathbf{1}}=\cG_{\mathbf{1}}\neq\emptyset$. Moreover, $$\|y\|_1=-y\cdot (-\mathbf{1})^T,\quad \forall y\in\cV_+$$  
Since $\cG$ is homogeneous, every source of $\cG_{-\mathbf{1}}$ is $-\mathbf{1}$-maximal in $\cV_{-\mathbf{1},+}$.
Then every reaction $y\ce{->[]} y'\in\cE_{-\mathbf{1}}$ is a $-\mathbf{1}$-endotacticity violating reaction of $\cG$ since \[(y'-y)\cdot (-\mathbf{1})^T=-(y'-y)\cdot\mathbf{1}^T>0\]

\noindent $(2)\Rightarrow(3)$. Note that $\cG$ has conservation law $\mathbf{1}$ immediately yields that $\|y\|_1=\|y'\|_1$ for every reaction $y\ce{->} y'\in\cE$. Hence $\cG$ is homogeneous consisting of first order reactions, and every complex is a single-copy species. Moreover, $d>1$ as otherwise $\cV$ is a singleton and it would have contradicted with $y\neq y'$ for any reaction $y\ce{->}y'\in\cE$.
\medskip

\noindent $(3)\Rightarrow(1)$. We prove $0\notin\cV$  by contradiction. Suppose $0\in\cV$. If $0\in\cV_+$, then   by homogeneity, $\cV_+=\{0\}$, and hence $\cG=\cG_{\mathbf{1}}=\cG_{-\mathbf{1}}$. By similar argument as in proving the implication $(1)\Rightarrow(2)$, one can show $\cV=\{0\}$, and this contradicts that $\cG$ is a reaction graph. Hence $0\in\cV\setminus\cV_+$,
i.e., $0$ is only a target. In this case, all reactions of $\cG$ are of first order, and thus all reactants are $-\mathbf{1}$-maximal in $\cV_+$. Since $0$ is a target, there must exist a reaction $y\ce{->}0\in\cE$ which is $-\mathbf{1}$-endotacticity violating, contradicting $-\mathbf{1}$-endotacticity of $\cG$. This contradiction yields the conclusion that $0\notin\cV$.
\end{proof}
\begin{remark}
 Without $\mathbf{1}$-endotacticity or $-\mathbf{1}$-endotacticity in Lemma~\ref{le:homogeneity}, homogeneity is \emph{insufficient} for $\cG$ to have conservation law $\mathbf{1}$. For instance, consider $$\cG\colon S_1\ce{->}0,$$ which is $\mathbf{1}$-endotactic while is not $-\mathbf{1}$-endotactic. Note that $\cG$ has no conservation law $\mathbf{1}$. Reversing the reaction of $\cG$ serves a counterexample when $\mathbf{1}$-endotacticity is lost.
\end{remark}

Recall that for $y,y'\in\cV$,  the notation  $y\rightharpoonup y'$ means that there exists a directed path from $y$ to $y'$ in a reaction graph. The lemma below characterizes $\cG$ when $\cG^0\neq\emptyset$. It will be used repeatedly (e.g., also in the proof of Theorem~\ref{thm:endotactic=weakly-reversible}).

\begin{lemma}\label{le:path-back-to-0}
Let $\cG=(\cV,\cE)$ be a first order reaction graph embedded in $\mathbb{N}^d_0$. Assume $\cG$ is $\cA$-endotactic and  $\cG^0\neq\emptyset$. 
Then \begin{equation}\label{eq:sources}
\cV_+^0=\{y\in\cV^0\colon \|y\|_1\le1\}
\end{equation} consists of the zero complex and single-copy species in $\cV^0$. Let  $$J=\{j\in[d]\colon\ e_j\rightharpoonup 0\},\ K=\supp\{y'\in\cV^0 \colon 0\rightharpoonup y'\},\ L=\{\ell\in J\setminus K\colon e_{k}\rightharpoonup e_{\ell}, \forall k\in K\}$$ Then
\begin{equation}\label{eq:support_of_non-zero_sources}
K\neq\emptyset;\quad K\cup L= J = \supp\cV^0
\end{equation}

 In other words, for every $j\in[d]$, there exists a path from $e_j$ to $0$ in $\cG^0$ if and only if either there exists a path from $0$ to a complex $y'\in\cV^0$ with $y'_j>0$ or there exists a path from $0$ to a complex $y'\in\cV^0$ with $y'_{k}>0$ and there exists a path from $e_{k}$ to $e_j$.
 In particular, there exists a path from every non-zero source in $\cV^0_+$ to $0$. 
Moreover, $\cG^0$ and $\cG^{\bullet}$ are subgraphs of $\cG$ of disjoint sets of species.
\end{lemma}
\begin{proof}
In the light of that $\cG^0$ is weakly connected,  \eqref{eq:sources} follows from
\eqref{eq:support_of_non-zero_sources}.

Note that $\cG^{\bullet}$ is a first order reaction graph without the zero complex. Moreover, since $\cG$ is $\mathbf{1}$-endotactic, and all nonzero sources $y$ are $\mathbf{1}$-maximal in $\cV_+$, which further implies that 
 \begin{equation}\label{eq:1-endotactic-inequality}
\|y'\|_1\le\|y\|_1=1,\quad \forall  y'\in\cV\ \text{such that}\ y\ce{->}y'\in\cE
\end{equation}
 In addition, each complex in $\cG^{\bullet}$ is a single copy of one species since its $\ell_1$-norm is 1. Hence $\cG^{\bullet}$  has conservation law $\mathbf{1}$ provided it were not empty. 

It follows from \begin{equation}\label{eq:supp}\supp\cV^0\cap\supp \cV^{\bullet}=\emptyset
\end{equation} that the set of species of $\cG^0$ and the set of species of $\cG^{\bullet}$ are disjoint. To see \eqref{eq:supp}, note that if \eqref{eq:supp} were false, then there exists $e_j\in\cV^1$ for some  $j\in \supp\cV^0=J$ by \eqref{eq:support_of_non-zero_sources}, and hence $e_j\rightharpoonup 0$ by the definition of $J$. This shows that $e_j\in\cV^0$ which leads to a contradiction.

It thus suffices to prove  \eqref{eq:support_of_non-zero_sources}.
 Let  \begin{equation}\label{eq:V-star-0}
\cV^0_*\coloneqq\{y\in\cV^0\colon \|y\|_1>1\}
\end{equation} Since $\cG$ is of first order, $\cV^0_*\subseteq\cV^0\setminus\cV^0_+$. By \eqref{eq:1-endotactic-inequality}, we have
\[0\ce{->}y,\quad \forall y\in\cV^0_*,\]
which yields that $\supp\cV^0_*\subseteq K$.

Then \eqref{eq:support_of_non-zero_sources} would follow from the four steps below. Step I. Prove  $K\neq\emptyset$. Step II: Prove $K\subseteq J$.  Step III: 
Prove $\supp\cV^0= J$. Step IV. Assume $J\setminus K\neq\emptyset$. Prove $L= J\setminus K$.
\medskip

\noindent \underline{Step I}. We prove $K\neq\emptyset$. Suppose $K=\emptyset$. Then  $0\in\cV^0\setminus\cV^0_+$, $\cV_*^0=\emptyset$,  and $\cG$ consists of first order reactions with all non-zero complexes being single-copy species.  Let $w=-\sum_{i=1}^de_i$. Then $e_i$ is $w$-maximal for all $i\in[d]$. By the definition of $J$, there exists a $j_0\in J$ such that $e_{j_0}\ce{->}0$. Note that $e_{j_0}\ce{->}0\in\cE_w$ is a $w$-endotacticity violating reaction since all sources in $\cV_{w,+}$ are $w$-maximal and
\[(0-e_{j_0})\cdot w^T=1>0\]
\medskip

\noindent \underline{Step II}. We prove $K\subseteq J$ by contradiction. By Step I, $K\neq\emptyset$, and hence  $0\in\cV_+^0$. We will show by contradiaction that otherwise there would exist a $w$-endotacticity violating reaction of $\cG$ for some $w\in\mathcal{A}$.  Suppose there exists $0\ce{->} y'\in\cE$ with $y_{k_0}'>0$ for some $k_0\in K\setminus J\subseteq\supp\cV\setminus J$.  Let $w=\sum_{i\in\supp\cV\setminus J}e_{i}$. We will show $0\ce{->}y'$ is a $w$-endotacticity violating reaction of $\cG$.

For any $z\in \cV_+\setminus (\{0\}\cup\{e_{j}\}_{j\in J})$, $z\ce{->[]} z'\in\cE$, by the definition of $J$, we have $$z'\in \cV\setminus (\{0\}\cup\{e_{j}\}_{j\in J});$$ additionally, $\|z\|_1=\|z'\|_1=1$. Hence $$(z'-z)\cdot w^T=\|z\|_1-\|z'\|_1=0$$ This shows $\cV_{w,+}\subseteq \{0\}\cup\{e_{j}\}_{j\in J}$. Since every complex in $\{0\}\cup\{e_{j}\}_{j\in J}$ is orthogonal to $w$, we know every source in $\cV_{w,+}$ is $w$-maximal in $\cV_{w,+}$. Then $0\to y'\in\cE_w$ with $0\in\supp_w\cG$ is a $w$-endotacticity violating reaction of $\cG$ since $$(y'-0)\cdot w^T=\sum_{i\in\supp\cV\setminus J}y'_{i}\ge y_{k_0}'>0$$ 
\medskip

\noindent \underline{Step III}. We prove $\supp\cV^0= J$ based on Steps I and II. Since the reverse inclusion $ J\subseteq\supp\cV^0$ holds trivially, it suffices to show
$\supp\cV^0\subseteq J$. We prove it by contradiction. Suppose $\supp\cV^0\setminus J\neq\emptyset$. Let $$\cV^{\RS{lir}}\coloneqq\{y\in\cV^0\colon \supp y\setminus J\neq\emptyset\}$$ Then for every $y\in\widetilde{\cV}$, there exists an  $i_0\in\supp\cV^0\setminus J$ such that $y_{i_0}>0$. Since  $K\subseteq J$, \begin{equation}\label{eq:no-path-connecting-0}
0\not\rightharpoonup y,\quad y\not\rightharpoonup 0
\end{equation}
\emph{While \eqref{eq:no-path-connecting-0} cannot imply $y\in\cV^1$ to reach a contradiction as two vertices within a weakly connected component may not be connected by a directed path in either direction,} it does yield that $y=e_{i_0}$, otherwise $y\in\cV^0_*$ and $0\rightharpoonup y$. This shows that $$\cV^{\RS{lir}}=\{e_{i}\}_{i\in\supp\cV^0\setminus J}$$
In other words, $\cV^{\RS{lir}}$ contains all single-copy species that does not connect to the zero complex. Since $\cV^0_*$ contains no source, we have $$z\rightharpoonup 0,\quad \forall z\in\cV^0\setminus(\cV^{\RS{lir}}\cup\cV^0_*\cup\{0\})\subseteq\{e_j\}_{j\in J}$$
In addition, since $K\subseteq J$ and each vertex in $\cV^0_*$ has in-degree one with the unique edge to that vertex in $\cG$ from $0$, we have $\supp (\cV^0\setminus\widetilde{\cV})= J$ and  $$\supp (\cV^0\setminus\widetilde{\cV})\cap\supp\widetilde{\cV}=\emptyset;$$ moreover, $\cG^0$ becomes ``bipartite'' in the sense that vertices in $\cV^{\RS{lir}}$ do not connect to those in $\cV^0\setminus\widetilde{\cV}$: $$y\not\rightharpoonup y',\quad \forall y\in\widetilde{\cV},\ \forall y'\in\cV^0\setminus\widetilde{\cV}$$
Hence by weak connectivity of $\cG^0$, there exists $e_{i_1}\in\widetilde{\cV}$ such that $z_*\ce{->}e_{i_1}$, for some $z_*\in\cV^0\setminus\widetilde{\cV}$. Note that $z_*\notin\{0\}\cup\cV^0_*$ and hence $z_*=e_{j_1}$ for some $j_1\in J$.  Let $w=\sum_{i\in \supp\cV^0\setminus J}e_{i}$. Since $$y\cdot w^T=1,\quad \forall y\in\widetilde{\cV},$$ we have $$\cE_w\subseteq\{y\ce{->}y'\in\cE\colon y\in\cV^0\setminus\widetilde{\cV}\},\quad \cV_{+,w}\subseteq\cV^0\setminus\widetilde{\cV}$$
which further implies by $\supp (\cV^0\setminus\widetilde{\cV}) = J$ that
\begin{equation}\label{eq:orthogonality}
y\cdot w^T=0,\quad \forall y\in\cV_{+,w}\end{equation} Moreover, for every $0\ce{->}y'\in\cE$, \[(y'-0)\cdot w^T=0-0=0,\] since $\supp y'\subseteq J$. Hence $0\notin\cV_{+,w}$, and every source $y\in\cV_{+,w}$ is $w$-maximal in $\cV_{+,w}$ owing to \eqref{eq:orthogonality}.  Indeed, $e_{j_1}\ce{->}e_{i_1}\in\cE_w$ is a $w$-endotacticity violating reaction since $$(e_{i_1}-e_{j_1})\cdot w^T=e_{i_1}\cdot w^T=1>0$$ This contradicts that $\cG$ is $\cA$-endotactic since $w\in\cA$.

\medskip

Based on Steps I-III, we have shown that $\supp\cV^0= J$, which immediately yields \begin{equation}\label{eq:sources-of-G^0}
\cV^0\setminus\cV^0_*=\cV^0_+=\{0\}\cup\{e_j\}_{j\in J}\end{equation}
\medskip

\noindent \underline{Step IV}. Assume  $J\setminus K\neq\emptyset$. We prove $L=J\setminus K$ 
by contradiction. 
Suppose $J\setminus (K\cup L)\neq\emptyset$. 
Let $w=-\sum_{j\in J\setminus  (K\cup L)}e_{j}\in\cA$. We will prove there exists a $w$-endotacticity violating reaction. Recall that based on $\supp\cV^0= J$, we have shown that $\cG^0$ and $\cG^{\bullet}$ have disjoint sets of species, which further implies that $\cE_w\subseteq\cE^0$.

Note that for any reaction $y\ce{->}y'\in\cE$, if $y\in\{0\}\cup\{e_i\}_{i\in K\cup L}$, then by the definitions of $K$ and $L$, we have $\supp y'\subseteq K\cup L$. Hence both $y$ and $y'$ are orthogonal to $w$. This yields that $y\ce{->}y'\notin \cE_w$, and thus $\cV_{w,+}\subseteq\{e_j\}_{j\in J\setminus (K\cup L)}$. Since all elements in $\{e_j\}_{j\in J\setminus (K\cup L)}$ are $w$-equal, every source in $\cV_{w,+}$ is $w$-maximal. On the other hand, since $e_j\rightharpoonup0$ for every $j\in J\setminus (K\cup L)$, one can show by induction that there exists a reaction $e_{j_0}\ce{->}y'\in\cE$ with $j_0\in J\setminus (K\cup L)$ and $y'=0$ or $y'=e_{i_0}$ for some $i_0\in K\cup L$. In either case, $e_{j_0}\ce{->}y'\in\cE_w$ since $$(y'-e_{j_0})\cdot w^T=0-(-1)=1>0$$ This further implies  $e_{j_0}\ce{->}y'$ is a $w$-endotacticity violating reaction since $e_{j_0}$ is $w$-maximal in $\cV_{w,+}$.

\medskip

Now we complete the proof.
\end{proof}

Note that $L$ can be a proper subset of $J$, as evidenced by the example below (see also Example~\ref{ex:WRDZ-realization}).
\begin{example}\label{ex:semi-connectivity}
 Consider the following reaction graph
\[\cG\colon S_1\ce{->}S_2\ce{->}0\ce{->}2S_1\]
Note that $\cG=\cG^0$ with $J=\{1,2\}$, $K=\{1\}$,  $L=\{2\}$, and the reaction graph is endotactic by \cite[Proposition~4.1]{CNP13}. In contrast, consider  similar reaction graphs
\begin{center}\begin{tikzcd}[row sep=1em, column sep = 1em]
\cG_1\colon S_1\rar{}&S_2\rar{}&0\rar{} &2S_2&& \cG_2\colon S_2\rar{}&0\arrow[bend left=15]{rr}{}&S_1\lar{}&2S_1
\end{tikzcd}
\end{center}
It is straightforward to see that for either reaction graph, $K\cup L\subsetneq J$, consistent with that neither of the reaction graphs is endotactic.
\end{example}
\begin{lemma}\label{le:first-order-reactions}
Let $\mathcal{G}=(\cV,\cE)$ be a non-empty  first order reaction graph  embedded in $\mathbb{N}^d_0$.  Assume $\cG$ is $\cA$-endotactic.  Then $\{e_i\}_{i\in[d]}\subseteq\mathcal{V}_+\subseteq\{0\}\cup\{e_i\}_{i\in[d]}$, and whenever $\cG^{\bullet}\neq\emptyset$,  $\cG^{\bullet}$ is homogeneous and WRDZ, and has conservation law $\mathbf{1}$ confined to its own set of  species. In particular, 
$\mathcal{V}_+=\{0\}\cup\{e_i\}_{i\in[d]}$ if and only if $0\in\cV$.
\end{lemma}
\begin{proof}
First, \emph{assume $\cG^{\bullet}$ is weakly reversible}. Since $\cG$ has no redundant species,  it follows from 
\eqref{eq:sources-of-G^0} that $$\{e_i\}_{i\in[d]}\subseteq\cV_+ = \cV^{\bullet}\cup\cV^0_+ \subseteq\{e_i\}_{i\in[d]\setminus J}\cup(\{0\}\cup\{e_j\}_{j\in J})=\{0\}\cup\{e_i\}_{i\in[d]},$$ 
In particular, $\mathcal{V}_+=\{0\}\cup\{e_i\}_{i\in[d]}$ if and only if 
$0\in\cV$. Moreover, it also follows from  Lemma~\ref{le:homogeneity} and Lemma~\ref{le:path-back-to-0} that whenever $\cG^{\bullet}\neq\emptyset$, $\cG^{\bullet}$ is homogeneous with conservation law $\mathbf{1}$ confined to its own set of species. Since $\cG^{\bullet}$ is monomolecular and weakly reversible, it is of deficiency zero.

It then remains to show weak reversibility of $\cG^{\bullet}$. We  will prove weak reversibility by contraposition. Since $\cG^0$ is weakly connected, assume w.l.o.g. that $\cG=\cG^{\bullet}$.  We further assume that $\cG$ is weakly connected; otherwise, the following arguments would apply to each of its weakly connected but not strongly connected components to yield a contradiction.
 Suppose $\cG$ is \emph{not} weakly reversible. Then $\cG$ can be decomposed into $k>1$ strongly connected components which are weakly connected.  Further assume w.l.o.g. (up to a permutation) that $\cG^1=(\cV^1,\cE^1)$ is a $\prec$-minimal component of $\cG$. Let $w=\sum_{e_{\ell}\in\cV\setminus\cV^1}e_{\ell}\in\cA$. Then  $\cE_w\neq\emptyset$ consists of reactions with one of the source and the target in $\cG^1$ and the other in a different strongly connected component that is weakly connected to $\cG^1$.  Furthermore, since $\cG^1$ is $\prec$-minimal, we have $\cV_{w,+}\subseteq\cV^1$.  Since $\cV$ consists of single-copy species due to Lemma~\ref{le:homogeneity}, it is straightforward to verify that every source in $\cV_{w,+}$ is $w$-maximal in $\cV_{w,+}$. This further implies that every reaction $y\ce{->}y'\in\cE_w$ is a  $w$-endotacticity violating reaction of $\cG$ since $$(y'-y)\cdot w^T=y'\cdot w^T=\|y'\|_1=1>0,$$
which contradicts  the $\cA$-endotacticity of $\cG$.
\end{proof}

Now we are ready to prove Theorem~\ref{thm:A-endotactic=endotactic}.

\begin{proof}
Assume w.l.o.g. that $\cG\neq\emptyset$. By Lemma~\ref{le:first-order-reactions}, $\cG^{\bullet}$ is WRDZ and hence endotactic; moreover, by Lemma~\ref{le:path-back-to-0}, $\cG^0$ and $\cG^{\bullet}$ have  disjoint sets of species. Since $\cG^0$ and $\cG^{\bullet}$ are disjoint sub reaction graphs of $\cG$, in the light of Lemma~\ref{le:endotacticity-preserved-under-joint}, it suffices to show $\cG^0$ is strongly endotactic provided it is non-empty. Assume w.l.o.g. that $\cG=\cG^0\neq\emptyset$. We will first show $\cG$ is endotactic and then show $\cG$ is strongly endotactic.

First, we prove endotacticity by contraposition. Suppose there exists a $w$-endotacticity violating reaction $y\ce{->[]} y'\in\cE$ for some $w\in\mathbb{R}^d$. Let $\cV_*^0$ be defined in \eqref{eq:V-star-0}. Then we have $\cV_*^0\neq\emptyset$. 
 Since otherwise, $\cG$ is monomolecular, and it follows from Lemma~\ref{le:path-back-to-0} that $J=K=[d]$ and  every non-zero complex connects to zero and vice versa, which implies that $\cG$ is strongly connected and hence is endotactic by Corollary~\ref{cor:simplified-criterion-for-endotacticity}.  We will achieve a contradiction in three steps.
\medskip

\noindent\underline{Step I}. We will show that $y=e_{i_0}$ for some $i_0\in[d]$ such that $w_{i_0}<0$ and $0\ce{->} y\notin\cE$. This will be achieved in three steps.

First, we will prove $y\neq0$ by repeatedly using contradiction argument. Suppose $y=0$. Then $y'\in\cV^0_*$. Since otherwise, $y'=e_{j_0}$ for some $j_0\in[d]$. By Lemma~\ref{le:path-back-to-0}, $y$ and $y'$ are in the same strongly connected component. This contradicts that $y\ce{->}y'$ is an endotacticity violating reaction due to Lemma~\ref{le:endotacticity-violating}.

Since $0\ce{->}y'$ is a $w$-endotacticity violating reaction, we have  $0<_wy'$, which implies $\supp y'\cap\supp_+w\neq\emptyset$.
We will show that this would contradict that $0$ is $w$-maximal in $\cV_{w,+}$. Choose $k_0\in\supp y'\cap \supp_+w$. Since $\supp y'\subseteq K$, by Lemma~\ref{le:path-back-to-0}, $e_{k_0}\rightharpoonup 0$. Hence by induction, one can show that there exists $j_0\in[d]$ such that $e_{j_0}\ce{->}0$, and either $e_{k_0}\rightharpoonup e_{j_0}$ or ${k_0}=j_0$. Since $0$ is $w$-maximal in $\cV_{w,+}$ while $e_{k_0}>_w0$, we know $e_{k_0}\notin\cV_{w,+}$. This implies by induction that $e_{k_0}=_we_{j_0}$, which further implies $j_0\in\supp_+w$ and hence $e_{j_0}>_w0$. This contradicts $0$ is $w$-maximal in $\cV_{w,+}$ since $e_{j_0}\ce{->}0\in\cE_w$.  So far we have shown that $y\neq0$. Hence $y=e_{i_0}$ for some $i_0\in[d]$.

Next, we will show $w_{i_0}<0$. By Lemma~\ref{le:path-back-to-0},  $y'\notin\cV^0_*$, which further implies by  Lemma~\ref{le:path-back-to-0} that either $y'=0$ or $y'\rightharpoonup0$.  By Lemma~\ref{le:endotacticity-violating} again,  in either case we have $e_{i_0}<_wy'=_w0$ which implies $w_{i_0}<0$.

Finally, we will prove that $0\ce{->} y\notin\cE$. Since $y\neq0$, we have $y'=0$ or $y'=e_{j_1}$ for some $j_1\in [d]=J$. By Lemma~\ref{le:path-back-to-0},  $e_{j}\rightharpoonup 0$ for all $j\in J$, which further implies that $y\rightharpoonup 0$. Suppose $0\ce{->} y\in\cE$. Then $y$ and $y'$ are in the same strongly connected component. This contradicts Lemma~\ref{le:endotacticity-violating}.

\medskip

\noindent\underline{Step II}. We will prove $w\le0$. Let $i_0$ be defined as in Step I. We will show $w\le0$ by contradiction. 
Suppose $\supp_+w\neq\emptyset$. We will show $e_{i_0}$ is not $w$-maximal in $\cV_{w,+}$ to achieve a contradiction. From Step I, we know $\supp_-w\neq\emptyset$. 
It follows from Lemma~\ref{le:path-back-to-0} that $e_j\rightharpoonup0$ for every $j\in[d]=J$. By induction, one can show that there exists $z\in\{e_{\ell}\}_{\ell\in\supp_+w}$ and $z'\in\{0\}\cup\{e_{\ell}\}_{\ell\in J\setminus\supp_+w}$ such that $z\ce{->}z'\in\cE$. Then \[(z'-z)\cdot w^T<0,\] which yields that $z\in\cV_{w,+}$. Note that $z>_we_{i_0}$ which implies that $e_{i_0}$ is not $w$-maximal in $\cV_{w,+}$. 
\medskip

\noindent\underline{Step III}. Let $i_0$ be defined as in Step I and $K$ defined as in Lemma~\ref{le:path-back-to-0}. We will show that $e_{i_0}$ is \emph{not} $w$-maximal in $\cV_{w,+}$, which would contradict that $y\ce{->}y'$ is a $w$-endotacticity violating reaction. It suffices to show $\{0\}\cup\{e_j\}_{j\in J\setminus\supp_-w}\cap\cV_{w,+}\neq\emptyset$ since $$y>_we_{i_0},\quad \forall y\in\{0\}\cup\{e_j\}_{j\in J\setminus\supp_-w}$$ From Step II, we know $\supp w=\supp_-w$.

If $K\cap\supp w\neq\emptyset$, then choose  $k_1\in K\cap\supp w$. By Lemma~\ref{le:path-back-to-0}, there exists $0\to z\in\cE$ with $z_{k_1}>0$. Based on Step II, $$(z-0)\cdot w^T=\sum_{j\in\supp z\cap\supp w}z_jw_j\le z_{k_1}w_{k_1}<0$$ This shows $0\to z\in\cE_w$ and hence $0\in\cV_{w,+}$.

If $K\cap\supp w=\emptyset$, by Lemma~\ref{le:path-back-to-0}, $i_0\in\supp w\subseteq L\neq\emptyset$, and $$e_k\rightharpoonup e_{\ell},\quad \forall \ell\in \supp w,\ \forall k\in K$$
One can show by induction that there exists $\tilde{y}\in\{0\}\cup\{e_j\}_{j\in J\setminus\supp w}$ and $i_1\in\supp w$  such that $\tilde{y}\ce{->}e_{i_1}\in\cE$. Note that
\[(e_{i_1}-\tilde{y})\cdot w^T=w_{i_1}<0,\]
which implies that $\tilde{y}\ce{->}e_{i_1}\in\cE_w$ and $\tilde{y}\in\cV_{w,+}$. 
\medskip

So far we have shown that $\cG$ is endotactic.

Next, we will show that $\cG$ is \emph{$w$-strongly endotactic} for every $w\in\mathbb{R}^d\setminus\mathsf{S}_{\cG}^{\perp}$. Since $\cG$ is $w$-endotactic, it suffices to show that $\cV_{w,+}$ contains a $w$-maximal source in $\cV_+$. We will prove it in two cases. Let $J,\ K,\ L$ be as defined in Lemma~\ref{le:path-back-to-0}. Recall from  Lemma~\ref{le:path-back-to-0} that $J=[d]$ and $\cV_+=\{0\}\cup\{e_j\}_{j\in J}$.
\medskip

\noindent \underline{Case I.} $\supp_+w\neq\emptyset$. Let $I=\{i\colon  w_i=\max\limits_{j\in\supp_+w} w_j\}$. Then $\{e_i\}_{i\in I}$ consist of all $w$-maximal sources in $\cV_+$. By Lemma~\ref{le:path-back-to-0} and induction,  we know there exists $e_{i_0}\in I$ and $y\in\{0\}\cup\{e_j\}_{j\in J\setminus I}$ such that $e_{i_0}\ce{->}y\in\cE$. Note that $e_{i_0}\ce{->}y\in\cE_w$ since \[(y-e_{i_0})\cdot w^T<0\] This shows that $e_{i_0}\in\cV_{w,+}$.
\medskip

\noindent \underline{Case II.} $\supp_+w=\emptyset$.  Since $w\neq0$, we have $\supp_-w\neq\emptyset$. Then $\{0\}\cup\{e_j\}_{j\in J\setminus\supp_-w}$ consist of all $w$-maximal sources in $\cV_+$. We will prove by contraposition that $\cV_{w,+}$ contains an element in $\{0\}\cup\{e_j\}_{j\in J\setminus\supp_-w}$. Suppose $\cV_{w,+}\cap \bigl(\{0\}\cup\{e_j\}_{j\in J\setminus\supp_-w}\bigr)=\emptyset$, i.e., $\cV_{w,+}\subseteq\{e_j\}_{\supp_-w}$. Since $0\notin\cV_{w,+}$, by the definition of $K$, we know $K\subseteq J\setminus\supp_-w$, which further implies that $K\neq J,\ \supp_-w\neq J$, and by Lemma~\ref{le:path-back-to-0}, $\supp_-w\subseteq L$. By Lemma~\ref{le:path-back-to-0} and induction,  we know that there exists $e_{k_0}\ce{->}e_{\ell_0}\in\cE$ for some $k_0\in K$ and $\ell_0\in \supp_-w$. Analogous to Case I, one can show that $e_{k_0}\ce{->}e_{\ell_0}\in\cE_w$ and $e_{k_0}\in\cV_{w,+}$, which contradicts with $\cV_{w,+}\cap \{e_j\}_{j\in J\setminus\supp_-w}=\emptyset$.

\end{proof}

\begin{remark}\label{re:A-endotactic=endotactic}
\begin{enumerate}
\item[(i)] For higher order endotactic reaction graphs $\cG$, \emph{homogeneity of $\cG$ may neither imply zero deficiency nor weak reversibility}. For instance, consider the second order reaction graph $2S_1\to S_1+S_2\ce{<-[]} 2S_2$, which is endotactic and homogeneous, but is of deficiency 2 and not weakly reversible. Nevertheless, it does have a WRDZ strong realization: $2S_1\ce{<=>[][]}2S_2$.
\item[(ii)] Different from the criterion \cite[Proposition~4.1]{CNP13} for $2$-dimensional   reaction graphs, the test set $\mathcal{A}$ for first order reaction graphs is \emph{independent} of the reaction graph.
\item[(iii)] $\cA$ seems to be a \emph{minimal} test set for endotacticity of first order reaction graphs, (at least for those of \emph{few} species) in the sense that for any proper subset $\widetilde{\cA}$ of $\cA$, $\widetilde{\cA}$-endotacticity does not imply endotacticity. For instance, for $d=1$, $\{1\}$ and $\{-1\}$ are the only two non-empty proper subsets of $\cA=\{1,-1\}$. However, $0\ce{->[]} S_1$ is $-1$-endotactic but not $1$-endotactic and $S_1\to 0$ is $1$-endotactic but not $-1$-endotactic. This shows $\cA$ is minimal in the above sense in this case. In general it could be \emph{non-trivial} to determine a minimal test set for endotacticity, even for first order reaction graphs.
\item[(iv)] Based on Theorem~\ref{thm:A-endotactic=endotactic} and Lemma~\ref{le:path-back-to-0}),  it has been proved in a companion work \cite{X24b} that \emph{every first order endotactic stochastic mass-action system is essential and exponentially ergodic}.
\end{enumerate}\end{remark}

Let $\cG$ be a first order reaction graph embedded in $\mathbb{N}^d_0$. Define $\cG^{\spadesuit}=(\cV^{\spadesuit},\cE^{\spadesuit})$ by $$\cE^{\spadesuit}=\cE^*\cup\{0\to S_k\}_{k\in K}\quad \text{and}\quad \cV^{\spadesuit}=\{y,y'\colon y\ce{->[]} y'\in\cE^{\spadesuit}\},$$  where  $K$ is defined as in  Lemma~\ref{le:path-back-to-0} for reaction graph $\mathcal{G}$.  Note that $\cG^{\spadesuit}$ is monomolecular and is of deficiency zero.

Next, we define strong realization of \emph{reaction graphs}.
\begin{definition}\label{def:weighted-graph-equivalence}
Let $\cG_i=(\cV_i,\cE_i)$ for $i=1,2$ be two non-empty  weighted  reaction graphs of the same set of sources. We say one reaction graph is a \emph{strong realization} of the other if \begin{equation}\label{eq:rate-constant-source-outflux}
\sum_{y\ce{->[]} y'\in\cE_1}\kappa_{y\ce{->[]} y'}^{(1)}(y'-y) = \sum_{y\ce{->[]} y'\in\cE_2}\kappa_{y\ce{->[]} y'}^{(2)}(y'-y),\quad \forall y\in\cV_+,
\end{equation} holds.
\end{definition}

\begin{definition}\label{def:unweighted-graph-equivalence}
Let $\cG_i=(\cV_i,\cE_i)$ for $i=1,2$ be two non-empty  \emph{unweighted}  reaction graphs of the same set of sources. We say
one reaction graph \emph{has the capacity to be a strong realization} of the other if there exist two collections of edge weights $\{\kappa_{y\ce{->}y'}^{(i)}\colon y\ce{->}y'\in\cE_i\}$ for $i=1,2$ such that one reaction graph is a strong realization of the other under these weights. With slight abuse of the term ``strong realization'', for two unweighted reaction graphs,
we say one is a \emph{strong realization} of the other for short if one has the capacity to be a \emph{strong realization} of the other.
\end{definition}

Before presenting the generic result on WRDZ strong realization of a first order endotactic reaction graph, let us revisit Example~\ref{ex:semi-connectivity} to have some intuition on why it makes sense to expect $\cG^{\spadesuit}$ to be WRDZ.
\begin{example}
Revisit Example~\ref{ex:semi-connectivity}:
\[\cG\colon S_1\ce{->}S_2\ce{->}0\ce{->}2S_1\]
By definition, 
\begin{center}\begin{tikzcd}[row sep=1em, column sep = 1em]
\cG^{\spadesuit}\colon S_1\rar{}&S_2\rar{}&0\arrow[bend left=15]{ll}
\end{tikzcd}
\end{center}
It is easy to observe it is a monomolecular weakly reversible reaction graph and hence is of deficiency zero. Moreover, by Definition~\ref{def:unweighted-graph-equivalence} and tuning the edge weight of $0\ce{->}S_1$ in $\cG^{\spadesuit}$ in accordance with that of $0\ce{->} 2S_1$ in $\cG$, it is straightforward to see that $\cG^{\spadesuit}$ is a strong realization of $\cG$.
\end{example}
\begin{theorem}\label{thm:endotactic=weakly-reversible}
Let $\cG$ be a  first order reaction graph embedded in $\mathbb{N}^d_0$. Then
\[\cG\ \text{is endotactic}\ \Leftrightarrow \cG^{\spadesuit}\ \text{is endotactic}\ \Leftrightarrow \cG^{\spadesuit}\ \text{is WRDZ}\]
Moreover, if $\cG\neq\emptyset$, then $\cG^{\spadesuit}$ is a strong realization of $\cG$.
\end{theorem}

\begin{proof}
Assume w.l.o.g. $\cG\neq\emptyset$. By Theorem~\ref{thm:A-endotactic=endotactic}, $\cG^{\bullet}=(\cG^{\spadesuit})^{\bullet}$ is a WRDZ monomolecular reaction graph. Moreover, $\cG^0$ and $\cG^{\bullet}$ have disjoint sets of species, and so do $(\cG^{\spadesuit})^{\bullet}$ and $(\cG^{\spadesuit})^0$.
In the light of Lemma~\ref{le:endotacticity-preserved-under-subtraction}, we assume w.l.o.g. that $\cG=\cG^0$. We first prove the two biimplications.
\medskip

\noindent(1) $\cG$ is endotactic $\implies$ $\cG^{\spadesuit}$ is WRDZ. This is a consequence of Theorem~\ref{thm:A-endotactic=endotactic}, Lemma~\ref{le:path-back-to-0}, as well as the fact that monomolecular weakly reversible reaction graphs are of deficiency zero.
\medskip

\noindent(2) $\cG^{\spadesuit}$ is WRDZ  $\implies$ $\cG^{\spadesuit}$ is endotactic. This is due to Corollary~\ref{cor:simplified-criterion-for-endotacticity}.
\medskip

\noindent(3) $\cG^{\spadesuit}$ is endotactic $\implies$ $\cG$ is endotactic. By the construction, $\cG^{\spadesuit}=(\cG^{\spadesuit})^{\spadesuit}$, and hence by the assumption that $\cG=\cG^0$, it follows from (1) and Theorem~\ref{thm:A-endotactic=endotactic} that $\cG^{\spadesuit}$ is strongly connected and strongly endotactic. In the light of Theorem~\ref{thm:A-endotactic=endotactic}, it suffices to show $\cG$ is $\cA$-endotactic. We prove it by contraposition.

Suppose $\cG$ is not $\cA$-endotactic, i.e., there exists $\emptyset\neq\cI\subseteq[d]$, $w_+=\sum_{i\in\cI}e_i$, and $w_-=-w_+$ such that either there exists a $w_+$-endotacticity violating reaction $y\ce{->[]} y'\in\cE$ of $\cG$, or there exists a $w_-$-endotacticity violating reaction $y\ce{->[]} y'\in\cE$ of $\cG$.

We first assume that there exists a $w_+$-endotacticity violating reaction $y\ce{->[]} y'\in\cE$ for $\cG$. Applying Lemma~\ref{le:first-order-reactions} to $\cG^{\spadesuit}$ yields $\cV^{\spadesuit}_+=\{0\}\cup\{e_j\}_{j\in[d]}$. Note that $\cV_+=\cV^{\spadesuit}_+$ by the construction of $\cG^{\spadesuit}$. Hence by Lemma~\ref{le:A-endotactic}(i), we have $$y\ce{->[]} y'\in\cE^{\spadesuit}_{w_+}\quad \text{and}\quad y\in\supp_{w_+}\cG=\supp_{w_+}\cG^{\spadesuit}\not\ni 0$$ This implies that $y\neq0$, and hence $y\ce{->[]} y'\in\cG^{\spadesuit}$ is also a $w_+$-endotacticity violating reaction of $\cG^{\spadesuit}$, contradicting strong-endotacticity of $\cG^{\spadesuit}$.

Analogously, one can also obtain a contradiction if we assume that there exists a $w_-$-endotacticity violating reaction $y\ce{->[]} y'\in\cE$ of $\cG$.
\medskip

Finally, by the construction of $\cG^{\spadesuit}$, it is straightforward to verify that $\cG^{\spadesuit}$ is a strong realization of $\cG$ with edge weights $\kappa_{y\ce{->[]} y'}'$ for every reaction $y\ce{->}y'\in\cE^{\spadesuit}$ chosen as follows:
\[\kappa_{y\ce{->[]} y'}'=\begin{cases}\kappa_{y\ce{->[]} y'},& \text{if}\ y\neq0,\\
\sum\limits_{0\to z\in\cE}z_k\kappa_{0\to z},& \text{if}\ y=0,\ y'=e_k,\ k\in K,\end{cases}\]
where $\kappa_{y\ce{->[]} y'}$ are the edge weights of $\cG$.
\end{proof}

\begin{corollary}\label{cor:WRDZ=monomolecular}
Any first order endotactic mass-action system is WRDZ if and only if it is monomolecular.
\end{corollary}
\begin{proof}
Let $\cG$ be the reaction graph associated with a first  order endotactic mass-action system.
\medskip

\noindent$\impliedby$ If $\cG$ is monomolecular, then $\cG=\cG^{\spadesuit}$, by Theorem~\ref{thm:endotactic=weakly-reversible}, $\cG$ is WRDZ.
\medskip

\noindent$\implies$ If $\cG$ is WRDZ, then $\cV=\cV_+$ and hence it is monomolecular.
\end{proof}

\begin{remark}
It is a classical while vibrant topic on computational aspect of CRN \emph{to determine if a given mass-action system $\cG$ has a weakly reversible mass-action system realization}   \cite{CP08,SHT12,RSHP14,CJY21,CDJ24,CJY23,D23}.  Theorem~\ref{thm:endotactic=weakly-reversible} identifies a class of reaction networks with WRDZ \emph{realization} (see Definition~\ref{def:dynamical_equivalence}).
\end{remark}

Theorem~\ref{thm:endotactic=weakly-reversible} provides an easily checkable criterion for endotacticity of first order reaction graphs.
\begin{example}\label{ex:WRDZ-realization}
Consider the first order reaction graph
\[\cG\colon S_3\to S_2\ce{->[]} S_1\to0\to 2S_3\]
Then
\begin{center}\begin{tikzcd}[row sep=1em, column sep = 1em]
\cG^{\spadesuit}\colon S_3\rar&S_2\rar&S_1\rar&0\arrow[bend left=15]{lll}
\end{tikzcd}
\end{center}
 is WRDZ. Hence $\cG$ is endotactic by Theorem~\ref{thm:endotactic=weakly-reversible}; moreover, it is strongly endotactic by Theorem~\ref{thm:A-endotactic=endotactic}.
\end{example}

\section{Stability of first order reaction systems}\label{sect:stability}

{As an application of the characterization of first order endotactic reaction graphs (Theorem~\ref{thm:A-endotactic=endotactic} and Theorem~\ref{thm:endotactic=weakly-reversible})}, we will show \emph{the global asymptotic stability of a unique equilibrium in each stoichiometric compatibility class for every first order endotactic mass-action system}.

\emph{Propensity function} $\lambda_{y\ce{->[]} y'}$ of a reaction $y\ce{->[]} y'\in\mathcal{G}$ is a non-negative function which quantifies the \emph{rate} at which a reaction fires. A family of propensity functions $\mathcal{K}\coloneqq\{\lambda_{y\ce{->[]} y'}\colon y\ce{->[]} y'\in\mathcal{E}\}$ defined on the common  domain $\mathbb{R}^d_+$ of a reaction graph $\mathcal{G}$ are called a \emph{deterministic kinetics} of $\mathcal{G}$. A \emph{non-empty} reaction graph $\cG$  with a deterministic kinetics $\mathcal{K}$ is  called a \emph{deterministic reaction system} and denoted $(\mathcal{G},\mathcal{K})$.  Since this paper only discusses deterministic dynamics of reaction networks, a deterministic reaction system is abbreviated as a \emph{reaction system}, or simply a \emph{system} whenever there is no confusion arising from the context.

Let us first recall the modeling of a reaction system. Given a reaction system $(\cG,\mathcal{K})$ with $\mathcal{K}=\{\lambda_{y\ce{->} y'}\colon y\ce{->} y'\in\mathcal{E}\}$, let $x(t)$ be the molar concentrations of species of $\cG$ at time $t$.  Then $x(t)$ solves the following  initial value problem (IVP) of an ODE:
\begin{equation}\label{eq:ODE}
\dot{x}(t) = \sum_{y\ce{->} y'\in\mathcal{E}}\lambda_{y\ce{->} y'}(x)(y'-y),\quad x(0)=x_0
\end{equation}
Hence every (positive) stoichiometric compatibility class is  an \emph{affine invariant subspace} of \eqref{eq:ODE}. Any equilibrium of \eqref{eq:ODE} is also called \emph{an equilibrium of the reaction system $(\cG,\mathcal{K})$}, or simply \emph{an equilibrium of $\cG$} when $\cK$ is apparent from the context.

Assume $(\cG,\mathcal{K})$ is a first order mass-action system embedded in $\mathbb{N}^d_0$. Then \eqref{eq:ODE} can be rewritten as
\begin{equation}\label{eq:ODE-linear}
\dot{x}=xA+b,\ x(0)=x_0
\end{equation}
where $A=(a_{ij})_{d\times d}\in\mathcal{M}_d(\mathbb{R})$ with $$a_{ij}=\sum_{e_i\to y'\in\mathcal{E}}\kappa_{e_i\to y'}(y_j'-y_j),\quad i,j=1,\ldots,d$$is called the \emph{average flux matrix} of $\cG$
and $b=(b_1,\ldots,b_d)$ with
\[b_i=\sum_{0\to y'\in\mathcal{E}}\kappa_{0\to y'}y_i',\quad i=1,\ldots,d\] is called the \emph{influx vector} of $\cG$. Note that $A^T$ is the Jacobian matrix associated with  \eqref{eq:ODE-linear}.

For $A\in\mathcal{M}_d(\mathbb{R})$, let
\[r(A)=\max\{\textrm{Re}\ \eta\colon \eta\ \text{is an eigenvalue of}\ A\}\]
be the \emph{spectral  abscissa} of $A$.
A matrix is \emph{Hurwitz stable} (\emph{Hurwitz semi-stable}, respectively) if $r(A)<0$ ($r(A)\le0$, respectively). A matrix is \emph{Metzler} if all of its off-diagonal entries are non-negative. \emph{For a Metzler matrix $A$, $r(A)$ is the largest real eigenvalue of $A$} (c.f., \cite[8.3.P9]{HJ13}. Since $\cG$ is embedded in $\mathbb{N}^d_0$,  it is easy to verify that $A$ is Metzler and $b$ is non-negative. 

\begin{proposition}\label{prop:endotactic-Hurwitz}
Let $\cG$ be a first order mass-action system, and $A$  be its  average flux matrix. If $\cG$ is $\mathbf{1}$-endotactic, then $r(A)\le0$. Moreover, assume $\cG$ is endotactic.  Then
$r(A)<0$ if and only if $\cG=\cG^0$.
\end{proposition}
\begin{proof}
We first show $r(A)\le0$. For every $i=1,\ldots,d$, we have
\begin{align*}
(\mathbf{1}A^T)_i =& e_iA\mathbf{1}^T\\
=& \sum_{y\ce{->[]} y'\in \mathcal{E}^*}\kappa_{y\ce{->[]} y'}e_i^y(y'-y)\cdot\mathbf{1}^T\\
=&\sum_{y\ce{->[]} y'\in \mathcal{E}^*,\ y=e_i}\kappa_{y\ce{->[]} y'}(y'-y)\cdot\mathbf{1}^T\\
=&\Bigl(\sum_{y\ce{->[]} y'\in \mathcal{E}^*}\kappa_{y\ce{->[]} y'}((y'-y)\cdot\mathbf{1}^T)y\Bigr)_i,
\end{align*}
i.e.,
\begin{equation*}
\mathbf{1}A^T= \sum_{y\ce{->[]} y'\in \mathcal{E}^*}\kappa_{y\ce{->[]} y'}((y'-y)\cdot\mathbf{1}^T)y
\end{equation*}
Since $\cG$ is $\mathbf{1}$-endotactic, it follows from \eqref{eq:1-endotactic-inequality} that 
$$(y'-y)\cdot\mathbf{1}^T\le0,\quad \forall y\ce{->[]} y'\in\mathcal{E}^*,$$ which implies that 
$\mathbf{1}A^T\le0$, i.e.,
\[0\le\sum_{j\neq i}a_{ij}\le -a_{ii},\quad \forall i\in[d]\]
By Ger\v{s}gorin Disc Theorem (c.f.,  \cite[Theorem~6.1.1]{HJ13}), $r(A)\le0$.

Next, we show the biimplication by contraposition. Assume $\cG$ is endotactic. By \cite[Theorem~2.5.3]{HJ94} (c.f. also \cite[Chapter 6, Theorem~2.3]{BP94}), in the light that $\cG^0$ and $\cG^{\bullet}$ have disjoint sets of species by Theorem~\ref{thm:A-endotactic=endotactic}, it suffices to show (1) $r(A)<0$ if $\cG=\cG^0$ and (2) $A$ has a zero eigenvalue if $\cG^{\bullet}\neq\emptyset$.
\medskip

\noindent(1) By Lemma~\ref{le:path-back-to-0}, for every $i\in[d]$,  $e_i\rightharpoonup0$. Hence there exists $e_j\ce{->}0$ for some $j\in[d]$ such that $e_i\rightharpoonup e_j$ or $i=j$. Then \begin{align*}
(\mathbf{1}A^T)_j = & \sum_{y\ce{->[]} y'\in\mathcal{E},\ y=e_j}\kappa_{y\ce{->[]} y'}(y'-y)\cdot\mathbf{1}^T\\
= & \kappa_{e_j\ce{->[]} 0}(0-e_j)\cdot\mathbf{1}^T=-\kappa_{e_j\ce{->[]} 0}<0,
\end{align*} which yields that $A$ is WCDD. By \cite[Theorem~2.1,Theorem~2.2]{BH64} (see also \cite{SC74};  \cite[Lemma~3.2]{AF16};  \cite[Corollary~6.2.27]{HJ13}), $A$ is a non-singular Metzler matrix and $A^{-1}$ is non-negative. By \cite[Theorem~2.5.3]{HJ94},  it further yields $r(A)<0$.
\medskip

\noindent(2) Let $\emptyset\neq I\subseteq[d]$ be the index set of the species of $\cG^{\bullet}$.  Let $w=\sum_{i\in\cI}e_i$. By Lemma~\ref{le:first-order-reactions}, $w$ confined to the set $I$ (as a $\#I$ dimensional vector) is a conservation law of $\cG^{\bullet}$, and due to $\cG^0$ and $\cG^{\bullet}$ has disjoint species sets, $A$ has a zero eigenvalue with a right eigenvector $w$.
\end{proof}
\begin{remark}\label{re:compartmental-system}
{It is noteworthy to mention a close connection to \emph{compartmental systems} \cite{JS93}. based on Proposition~6.1, one can show that, \emph{every first order $\mathbf{1}$-endotactic mass-action system $\mathcal{G}$ is a compartmental system}, where the transpose $A^T$ of the average flux matrix $A$ is referred to as the \emph{compartmental matrix} and $b$ the forcing function \cite[Section~7B]{A83}. 
In addition, $\mathcal{G}^{0}$, if it is non-empty, induces an open \emph{outflow-connected}\footnote{A compartmental system is said to be \emph{outflow-connected} if there is a path from every compartment to a compartment which has outflow to the environment \cite{JS93}. Note that by definition, $J$ is the subset of compartments from which there is a path to a compartment having outflow to the environment.} compartmental system of no trap with a Hurwitz stable WCDD compartmental matrix\footnote{{A compartmental matrix $A$ is WCDD if and only if it is non-singular if and only if there exists a positive diagonal matrix $D$ such that $AD$ is SDD \cite{JS93,AF16}.}} \cite[Theorem~3]{JS93}. In contrast, $\mathcal{G}^{\bullet}$ is an outflow-closed compartment system consisting of finitely many {simple traps}\footnote{Each strongly connected component in $\mathcal{G}^{\bullet}$ is called a \emph{simple trap} \cite{JS93}.} \cite{JS93}. Hence $r(A)\le0$, $A$ has no purely imaginary eigenvalues \cite[Theorem~12.1]{A83}, and the multiplicity of the zero eigenvalue  of $A$ equals the number of strongly connected components in $\mathcal{G}^{\bullet}$ \cite{FJ75} (c.f. \cite[Theorem~2]{JS93}).} 
\end{remark}

Despite $\mathbf{1}$-endotacticity of $\cG$ is enough for $r(A)<0$ to imply $\cG=\cG^0$, it is \emph{insufficient} for the reverse implication.
\begin{example}
Consider
\[\cG\colon 0\ce{<-[$\kappa_1$]}S_1\ce{->[$\kappa_2$]}S_2\]
It is readily verified that $\cG=\cG^0$ is $(1,1)$-endotactic but \emph{not} endotactic by Theorem~\ref{thm:endotactic=weakly-reversible} since $\cG^{\spadesuit}=\cG$ is not WRDZ. However \[A=\begin{bmatrix}
-\kappa_1-\kappa_2&\kappa_2\\
0&0
\end{bmatrix}\]
and hence $r(A)=0$.
\end{example}

Next, we prove \emph{exponential} global asymptotic stability of the unique non-negative equilibrium (see Theorem~\ref{thm:exisitence-of-positive-equilibrium} for the existence and the explicit formula of the equilibrium). For a first order endotactic mass-action system $\cG$, to provide an \emph{accurate rate of exponential convergence} particularly when $\cG^{\bullet}\neq\emptyset$, let $n$ be the maximum of numbers of \emph{sources} of each weakly connected component of $\cG$, and define
\[\rho=-\max\{\text{Re}\ \lambda\colon \lambda\ \text{is a non-zero eigenvalue of } A\}\]
If $\cG^{\bullet}\neq\emptyset$, let
$$\gamma(\cG^{\bullet})=\min\{\text{Re}\ \lambda\colon \lambda\ \text{is a non-zero eigenvalue of } \mathsf{L}(\cG^{\bullet})\}$$
In this case, it is readily verified that $$\rho=-\min\{r((-(\mathsf{L}(\cG^0))_{ij})_{i,j\in\cI_0}),-\gamma(\cG^{\bullet})\},$$
where $\cI_0$ is the set of indices of species of $\cG^0$.

If $\cG^{\bullet}=\emptyset$, then $n=d+1$ and $\rho=-r(A)$.

Before presenting the global asymptotic stability of the non-negative equilibrium with a \emph{sharp} rate of exponential conergence, we first provide an intuitive example.
\begin{example}\label{ex:rate-convergence}
Revisit the mass-action system in Example~\ref{ex:degenerate-class}: 
\begin{center}\begin{tikzcd}[row sep=1em, column sep = 1em]
\cG\colon S_2\rar{2}&S_1\rar{2}&0\rar{2}&S_1+S_2&S_3\rar{1}&S_4\rar{1}&S_5\arrow[bend left=15]{ll}{2}
\end{tikzcd}
\end{center}
It is easy to verify that the average flux matrix $A=\begin{bmatrix}A_1&0\\
0&A_2\end{bmatrix}$ is block diagonal, with $$A_1=(-(\mathsf{L}(\cG^0))_{ij})_{i,j\in[2]}=\begin{bmatrix}-2&0\\
2&-2\end{bmatrix},\quad A_2=-\mathsf{L}(\cG^{\bullet})=\begin{bmatrix}-1&1&0\\
0&-1&1\\
2&0&-2\end{bmatrix}$$ Note that $A_2$ is indeed a 
\emph{$Q$-matrix of an irreducible CTMC}  $Z_t$ on a 3-state space $\{S_3,S_4,S_5\}$ with the unique stationary distribution $\pi=(\tfrac{2}{5},\tfrac{2}{5},\tfrac{1}{5})$. Moreover, for every $a\ge0$,  $x_{*,a}=(y_*,z_{*,a})$ with $y_{*}=(2,1)$ and 
$z_{*,a}=a\pi$ is the unique equilibrium in a stoichiometric compatibility class $\Gamma_{a\pi}=\mathbb{R}^2_+\times a\Delta_3$. In addition,  the eigenvalues of $A_1$ are $-2$ of multiplicity 2 and the eigenvalues of $A_2$ are $0$, $-2+\textrm{i}$, and $-2-\textrm{i}$. Hence $n=3$, $r(A_1)=-2$, $\gamma(\cG^{\bullet})=2$, and $\rho=2$.

Let $x(t)=(y(t),z(t))$ be the solution to the  ODE associated with $\cG$ subject to the initial concentration $x_0=(y_0,z_0)$. Using variation of constants formula, by straightforward computation, it is easy to obtain that
\[y(t)-y_* = (y_0-y_*)e^{A_1t} = (y_0-y_*)e^{-2 t}
\begin{bmatrix}
 1 & 2 t\\
 0 & 1
\end{bmatrix}\]
\begin{align*}
z(t)-z_{*,\|z_0\|_1} =& z_{0}(e^{A_2t}-(1,1,1)^T\pi)\\
=& z_{0}e^{-2 t}\begin{bmatrix}
 \frac{1}{5}  \sin t+\frac{3}{5}  \cos t & \frac{1}{5}  \sin t-\frac{2}{5} \cos t& -\frac{2}{5} \sin t-\frac{1}{5}  \cos t\\
 -\frac{4}{5} \sin t-\frac{2}{5}\cos t& \frac{1}{5} \sin t+\frac{3}{5} \cos t& \frac{3}{5} \sin t-\frac{1}{5}  \cos t\\
 \frac{6}{5} \sin t-\frac{2}{5} \cos t & -\frac{4}{5}  \sin t-\frac{2}{5} \cos t& -\frac{2}{5}  \sin t+\frac{4}{5} \cos t
\end{bmatrix}
\end{align*}
This further yields that \begin{equation}\label{eq:expression_of_solution}
\begin{split}
&\|x(t)-x_{*,\|z_0\|_1}\|_1 \\
= & e^{-2t}|x_{0,1}-2|+e^{-2t}|(x_{0,1}-2)2t+(x_{0,2}-1)|\\
& +e^{-2t}\Bigl|x_{0,3}(\tfrac{1}{5}\sin t+\tfrac{3}{5}\cos t)+x_{0,4}(-\tfrac{4}{5}\sin t-\tfrac{2}{5}\cos t)+x_{0,5}(\tfrac{6}{5}\sin t-\tfrac{2}{5}\cos t)\Bigr|\\
&+e^{-2t}\Bigl|x_{0,3}(\tfrac{1}{5}\sin t-\tfrac{2}{5}\cos t)+x_{0,4}(\tfrac{1}{5}\sin t+\tfrac{3}{5}\cos t)+x_{0,5}(-\tfrac{4}{5}\sin t-\tfrac{2}{5}\cos t)\Bigr|\\
&+e^{-2t}\Bigl|x_{0,3}(-\tfrac{2}{5}\sin t-\tfrac{1}{5}\cos t)+x_{0,4}(\tfrac{3}{5}\sin t-\tfrac{1}{5}\cos t)+x_{0,5}(-\tfrac{2}{5}\sin t+\tfrac{4}{5}\cos t)\Bigr|\\
\le&e^{-2t}g(t),
\end{split}
\end{equation}
where $$g(t)=\max\{\|x_0\|_1,\|x_0-x_*\|_1\}(\tfrac{6\sqrt{2}}{5}+2t)$$ is a linear function in $t$.
\end{example}

\begin{theorem}\label{thm:deterministic-stability-endotactic}
Let $\cG$ be a first order endotactic mass-action system. Given  a stoichiometric compatibility class  $\Gamma$ of $\cG$, let $x_{*,\Gamma}$ be the unique equilibrium on $\Gamma$ and $x(t)$ be the solution to the ODE \eqref{eq:ODE-linear} with $x_0\in\Gamma$. Then there exists a polynomial $g$ of degree $\le n-2$ which depends on $x_0$ such that
\[\|x(t)-x_{*,\Gamma}\|_1\le g(t) e^{-\rho t},\quad \forall t\ge0\]
\end{theorem}
\begin{proof}
{The asymptotic stability of the equilibrium immediately follows from Remark~\ref{re:compartmental-system}. For self-containedness, we provide a proof independent of the theory of compartmental systems.}

For the same sake as in the proof of Theorem~\ref{thm:exisitence-of-positive-equilibrium}, it suffices to prove the following two special cases.
\medskip

\noindent\underline{Case I.} Assume $\cG=\cG^0$. It follows from Proposition~\ref{thm:linear-ode}.
\medskip

\noindent\underline{Case II.} Assume $\cG=\cG^{\bullet}$    is strongly connected. Then $b=0$ and $A$ is a $Q$-matrix which defines a finite irreducible CTMC  $Z_t$ on the state space $\cV$ (i.e., a CTMC on the graph $\cG$). Then $p(t)\coloneqq\frac{x(t)}{\|x_0\|_1}$ is the probability distribution of $Z_t$ given the initial distribution $p_0=\frac{x_0}{\|x_0\|_1}$. The  exponential global asymptotic stability of $x(t)$ confined to $\Gamma$ follows from the \emph{uniform exponential ergodicity} of $Z_t$, where the precise upper estimate of $\|p(t)-p_0\|_1$ comes from a continuous-time analogue of the classical exponential ergodicity result for discrete time Markov chains (e.g., \cite[Theorem~5.3]{J13}).
\end{proof}

\begin{remark}
\begin{enumerate}
\item[(i)] Complex dynamics may emerge for higher order mass-action systems. For instance, 2-species second order mass-action systems were constructed to undergo \emph{fold bifurcations}, \emph{Hopf bifurcations}, \emph{Bogdanov-Takens bifurcations}, and \emph{Bautin bifurcations} \cite{BBH24} (see also \cite{BB23,JTZ24} for bifurcations and multistability of mass-action systems). 
In the light of Theorem~\ref{thm:deterministic-stability-endotactic}, these complex dynamics will \emph{not} appear in first order endotactic mass-action systems, despite linear ODEs 
generically allow for dynamics such as Hopf bifurcations.
\item[(ii)] It follows from Theorem~\ref{thm:deterministic-stability-endotactic}  a stronger persistence, the so-called \emph{``vacuous persistence''} \cite{G09,G11}, which means that trajectories starting even from the \emph{boundary} of a positive stoichiometric compatibility class 
 will eventually keep a positive distance from the boundary. Certain binary enzymatic networks were shown to be vacuously persistent \cite{G09,G11}.
\item[(iii)] Revisit Example~\ref{ex:rate-convergence}.  It follows from the equality in \eqref{eq:expression_of_solution}  that $$\|x(t)-x_*\|_1 \ge|x_{0,1}-2|g_*(t)e^{-2t},$$
where for all $x_{0,1}\neq 2$, $g_*(t)=1+|2t+\tfrac{x_{0,2}-1}{x_{0,1}-2}|$ is a linear function in $t$. Recall that $\rho=2$ and $n=3$. This example illustrates that both the exponential rate and the degree of the polynomial $g$ in the upper estimate in  Theorem~\ref{thm:deterministic-stability-endotactic} are  \emph{sharp}.
\end{enumerate}
\end{remark}
Let $(\cG,\cK)$ be a reaction system, where $\cG=(\cV,\cE)$ and $\cK=\{\lambda_{y\ce{->}y'}\colon y\ce{->}y'\in\cE\}$. Recall that $(\cG,\cK)$ is \emph{complex balanced} if there exists an equilibrium  $x_*\in\mathbb{R}^d_+$ such that
 \begin{equation}\label{eq:complex-balancing}
\sum_{y\ce{->}y'\in\cE} \lambda_{y\ce{->}y'}(x_*) = \sum_{y'\ce{->}y\in\cE} \lambda_{y'\ce{->}y}(x_*),\quad \forall y\in\cV
 \end{equation}
 Any equilibrium $x_*\in\mathbb{R}^d_+$ satisfying \eqref{eq:complex-balancing} is called a \emph{complex balanced equilibrium} of $(\cG,\cK)$.
\begin{corollary}\label{cor:GAC-linear-complex-balanced-system}
Every linear complex balanced mass-action system has a globally attractive positive equilibrium in each positive stoichiometric compatibility class.
\end{corollary}
\begin{proof}
Let $\cG$ be a linear complex balanced mass-action system. Then $\cG$ is a first order weakly reversible mass-action system \cite[Theorem~3C]{H72}. From Theorem~\ref{thm:exisitence-of-positive-equilibrium}, there exists a positive equilibrium $x_{*,\Gamma}$ on each positive stoichiometric compatibility class $\Gamma$ of $\cG$. By Theorem~\ref{thm:deterministic-stability-endotactic}, $x_{*,\Gamma}$ is globally attractive on $\Gamma$.
\end{proof}
\begin{remark}
{Observe that every first order weakly reversible reaction graph can be decomposed into strongly connected components of pairwise disjoint species. Based on this observation, Corollary~\ref{cor:GAC-linear-complex-balanced-system} can be proved by applying \cite{A11b,A11a} (or \cite{GMS14}) to each sub reaction system with a strongly connected reaction graph.}
\end{remark}
Finally, we demonstrate the potential applicability of the stability result to higher order reaction systems modeled by \emph{nonlinear} ODEs,  via \emph{the theory of asymptotically autonomous differential equations} \cite{T92}.
\begin{example}
\label{ex:nonlinear}
Consider the following bimolecular mass-action system:
\begin{center}\begin{tikzcd}[row sep=1em, column sep = 1em]
\cG\colon S_2\rar{2}&S_1\rar{2}&0\rar{2}&S_1+S_2& S_2+S_3\rar{3}&S_1+S_3&S_3\rar{1}&S_4\rar{1}&S_5\arrow[bend left=15]{ll}{2}
\end{tikzcd}
\end{center}

It is easy to observe that $S_2+S_3\ce{->[3]}S_1+S_3$ is a $(1,0,2,2,2)$-endotacticity violating reaction of $\mathcal{G}$, and hence   $\mathcal{G}$ is not endotactic. The associated ODE with this mass-action system is given by
\begin{equation}\label{eq:original_ODE}
    \dot{y}= y(A_1+\widetilde{A}_1(z^{(1)})) + (2,2),\quad \dot{z} = z A_2,
\end{equation}
where $x=(y,z)$, $z=(z^{(1)},z^{(2)},z^{(3)})$, $A_1$, $A_2$ are given in Example~\ref{ex:rate-convergence}, and  $\widetilde{A}_1(z^{(1)}) = \begin{bmatrix}
    0&0\\
3z^{(1)}&-3z^{(1)}
\end{bmatrix}$. Let $V(x)=y_1+2y_2+\|z\|_1$. Using $V$ as a Lyapunov function, one can show that there exists a \emph{compact} global attractor in each stoichiometric compatibility class $\Gamma_{a\pi}=\mathbb{R}^2_+\times a\Delta_3$ for every $a>0$. This implies that all forward  solutions of \eqref{eq:original_ODE} in $\Gamma_{a\pi}$ are bounded. Note that for all initial conditions $x_0=(y_0,z_0)\in \Gamma_{a\pi}$, $z(t)\to z_{*,a}=a(\tfrac{2}{5},\tfrac{2}{5},\tfrac{1}{5})$ as $t\to\infty$. Hence the right hand side of $\dot{y}$ in  \eqref{eq:original_ODE} converges to $y(A_1+\widetilde{A}_1(\tfrac{2}{5}a)) + (2,2)$ as $t\to\infty$, and \eqref{eq:original_ODE} becomes  \emph{asymptotically autonomous} \cite{T92} \emph{with limit equation} 
\begin{equation}
\label{eq:limit_ODE}
\dot{y} = y(A_1+\widetilde{A}_1(\tfrac{2}{5}a)) + (2,2),\quad  \dot{z} = z A_2,
\end{equation}
which corresponds to the following \emph{first order endotactic} mass-action system:
\begin{center}\begin{tikzcd}[row sep=1em, column sep = 1em]
\cG_{\square}\colon S_2\arrow[]{rrr}{2+3\cdot\tfrac{2}{5}a}&&&S_1\rar{2}&0\rar{2}&S_1+S_2&S_3\rar{1}&S_4\rar{1}&S_5\arrow[bend left=15]{ll}{2}
\end{tikzcd}
\end{center}
By Theorem~\ref{thm:deterministic-stability-endotactic}, $\mathcal{G}_{\square}$ has a globally asymptotically stable positive equilibrium $x_{*,a}$ in the \emph{common}  stoichiometric compatibility class $\Gamma_{a\pi}$. It is readily verified that $x_{*,a}$ is also the unique equilibrium of \eqref{eq:original_ODE}. As a consequence of \cite[Theorem~1.2]{T92} due to Markus \cite{M56} for asymptotically autonomous differential equations, $x_{*,a}$ is also the unique globally asymptotically stable equilibrium of $\mathcal{G}$ in $\Gamma_{a\pi}$.  
\end{example}
\begin{remark}
    From Example~\ref{ex:nonlinear}, we can extend the global stability result (Theorem~\ref{thm:deterministic-stability-endotactic}) to mass-action systems which are \emph{asymptotically autonomous with limit equation as a first order endotactic mass-action system}.
\end{remark}

%\section*{Declarations}

%\subsection*{Data availability} Data sharing is not applicable to this article as no datasets were generated or analyzed in this study.

%\subsection*{Conflict of interest} The author has no competing interests to declare that are relevant to the content of this article.

\appendix

\section{Kinetics}\label{subsect:kinetics}

We introduce a typical type of kinetics  
which is prevalently presumed in chemistry \cite{T90}.
\begin{definition}\label{def:source-dependent-kinetics}
A deterministic kinetics is called a \emph{source-dependent kinetics} (SDK) if the propensity function of each reaction in the reaction network is proportional to a non-negative function that only relies on the source of the reaction. In this case, a reaction system is called an \emph{SDK system}.

Hence an SDK system $(\cG,\cK)$ can be represented by a \emph{weighted reaction graph} $\cG$ and a collection of non-negative \emph{generating  propensity functions} $\mathcal{F}=\{\eta_y\colon \mathbb{R}^d_+\ce{->[]} \mathbb{R}_+\}_{y\in\cV_+}$ such that
\[\lambda_{y\ce{->[]} y'}(x) = \kappa_{y\ce{->[]} y'}\eta_y(x),\quad \forall y\ce{->[]} y'\in\cE,\]
where $\kappa_{y\ce{->[]} y'}$, the edge weight of $y\ce{->[]} y'$, is called the \emph{reaction rate constant} of the reaction $y\ce{->[]} y'$. We call the quantity $\sum_{y\ce{->[]} y'\in\mathcal{E}}\lambda_{y\ce{->[]} y'}(x)(y'-y)$ the \emph{average kinetic flux rate} of $(\cG,\cK)$.
\end{definition}
\begin{definition}\label{def:dynamical_equivalence}
Let $(\cG_i,\mathcal{K}_i)$ for $i=1,2$ be two reaction systems with  $\cG_i=(\cV_i,\cE_i)$ and $\mathcal{K}_i=\{\lambda_{y\ce{->[]} y'}^{(i)}\colon y\ce{->[]} y'\in\mathcal{E}_i\}$. We say $(\cG_1,\mathcal{K}_1)$ and $(\cG_2,\mathcal{K}_2)$ are \emph{dynamically equivalent} if they share the same average kinetic flux rate:
\begin{equation}\label{eq:overall-flux}
\sum_{y\ce{->[]} y'\in\mathcal{E}_1}\lambda^{(1)}_{y\ce{->[]} y'}(x)(y'-y)= \sum_{y\ce{->[]} y'\in\mathcal{E}_2}\lambda^{(2)}_{y\ce{->[]} y'}(x)(y'-y),\quad \forall x\in\mathbb{R}^d_+;
\end{equation} in this case, one reaction system is called a  \emph{realization} of the other. \end{definition}

 Dynamical equivalence is also called \emph{``macro-equivalence''} \cite{HJ72} or \emph{``confoundability''} \cite{CP08}, in the literature of CRNT. The generic phenomenon that a given ODE may associate with  different reaction systems is the so-called \emph{non-identifiability} of reaction systems \cite{CP08}.

\begin{definition}\label{def:strong-realization}
Let $(\cG_i,\cK_i)$ for $i=1,2$ be two reaction systems with $\cG_i=(\cV_i,\cE_i)$ and $\cK_i=\{\lambda^{(i)}_{y\ce{->[]} y'}\colon y\ce{->[]} y'\in\cE_i\}$. Assume the two reaction systems share the same set of sources $\cV_+$. We say one reaction system is a \emph{strong realization} of the other if their \emph{kinetic flux rates} are identical at each source\begin{equation}\label{eq:source-outflux}
\sum_{y\ce{->[]} y'\in\cE_1}\lambda^{(1)}_{y\ce{->[]} y'}(x)(y'-y) = \sum_{y\ce{->[]} y'\in\cE_2}\lambda^{(2)}_{y\ce{->[]} y'}(x)(y'-y),\quad \forall y\in\cV_{+},\quad \forall x\in\mathbb{R}^d_+
\end{equation}
\end{definition}

Hence every strong realization of a reaction system is a realization of that system.

For two SDK systems of the same set of sources as well as the same collection of generating propensity functions, the condition~\eqref{eq:source-outflux} in Definition~\ref{def:strong-realization} can be rephrased  as \eqref{eq:rate-constant-source-outflux}, 
which is purely a relation between the two \emph{weighted} reaction graphs $\cG_1=(\cV_1,\cE_1)$ and $\cG_2=(\cV_2,\cE_2)$ with respective edge weights $\{\kappa_{y\ce{->}y'}^{(i)}\colon y\ce{->}y'\in\cE_i\}$ for $i=1,2$.

The following is a direct consequence of Definition~\ref{def:strong-realization} and Definition~\ref{def:weighted-graph-equivalence}.
\begin{proposition}
Let  $(\cG_i,\cK_i)$ for $i=1,2$ be two SDK systems with  the same set of sources and the same collection $\cF$ of  generating propensity functions. Then $(\cG_2,\cK_2)$ is a strong realization of $(\cG_1,\cK_1)$ if and only if $\cG_2$ is a strong realization of $\cG_1$.
\end{proposition}

Next for a given SDK system, we specify when every realization is a strong realization, in terms of the generating propensity functions.
\begin{proposition}\label{prop:linear-independence-realization}
Let  $\cG_i$ for $i=1,2$ be two SDK systems with the same set of sources and the same collection $\cF$ of generating propensity functions. Assume $\cF$ consists of linearly independent functions, i.e., $\dim\spa\cF=\#\cF$, where $$\spa\cF=\{\sum_{j=1}^mc_jf_j\colon c_j\in\mathbb{R},\ f_j\in\cF,\ j=1,\ldots,m\}$$ Then $(\cG_2,\cK_2)$ is a strong realization of $(\cG_1,\cK_1)$ if and only if $(\cG_2,\cK_2)$ is a realization of $(\cG_1,\cK_1)$.
 \end{proposition}
 \begin{proof}
\noindent $\implies$ This is obvious by definition.
\medskip

\noindent $\impliedby$ Let $\cV_+$ be the set of sources and $\cF=\{\eta_y\colon y\in\cV_+\}$. Note that \eqref{eq:overall-flux} can be rewritten as
 \[\sum_{y\in\cV_+}\Bigl(\sum_{y\ce{->[]} y'\in\cE_1}\frac{\lambda^{(1)}_{y\ce{->[]} y'}(x)}{\eta_{y}(x)}(y'-y) -\sum_{y\ce{->[]} y'\in\cE_2}\frac{\lambda^{(2)}_{y\ce{->[]} y'}(x)}{\eta_{y}(x)}(y'-y)\Bigr) \eta_{y}(x)= 0,\quad \forall x\in\mathbb{R}^d_+,\quad \eta_{y}(x)>0\]
 By linear independence of $\cF$, it yields
 \[\sum_{y\ce{->[]} y'\in\cE_1}\frac{\lambda^{(1)}_{y\ce{->[]} y'}(x)}{\eta_{y}(x)}(y'-y) -\sum_{y\ce{->[]} y'\in\cE_2}\frac{\lambda^{(2)}_{y\ce{->[]} y'}(x)}{\eta_{y}(x)}(y'-y)=0,\quad \forall y\in\cV_+,\ \forall x\in\mathbb{R}^d_+,\quad \eta_{y}(x)>0,\]
 which further implies  \eqref{eq:source-outflux} by multiplying $\eta_y$ on both sides.
 \end{proof}

A popular SDK is the \emph{deterministic mass-action kinetics}, which is given by $$\eta_y(x)=x^y,\quad \forall  y\in\cV_+,\quad \forall x\in\mathbb{R}^d_+$$
A reaction system with deterministic mass-action kinetics is also known as a \emph{mass-action system} in the literature of CRNT \cite{H72,F72}.

We next introduce the \emph{joint} of two reaction systems.

\begin{definition}\label{def:joint-of-KRN}
Let $(\cG_1,\mathcal{K}_1)$ and $(\cG_2,\mathcal{K}_2)$ be two reaction systems, where $\mathcal{K}_i=\{\lambda^{(i)}_{y\ce{->[]} y'}\colon y\ce{->[]} y'\in\mathcal{E}_i\}$ for $i=1,2$. We define their \emph{joint} $(\cG_1\cup\cG_2,\cK_{1,2})$ as a reaction system with the kinetics $\cK_{1,2}=\{\lambda_{y\ce{->[]} y'}^{(1,2)}\colon y\ce{->[]} y'\in\cE_1\cup\cE_2\}$ given by
$$\lambda^{(1,2)}_{y\ce{->[]} y'}(x)=\lambda^{(1)}_{y\ce{->[]} y'}(x)\mathbbm{1}_{\cE_1}(y\ce{->[]} y') +\lambda^{(2)}_{y\ce{->[]} y'}(x)\mathbbm{1}_{\cE_2}(y\ce{->[]} y'),\quad x\in\mathbb{R}^d,$$  where $\cG_1\cup\cG_2$ with the set of edges $\cE_1\cup\cE_2$
is the joint reaction graph as defined in Definition~\ref{def:joint-of-RN}.
\end{definition}
 \begin{remark}
The joint of reaction systems has been studied in the  literature of CRNT  \cite{GHMS20}, e.g., motivated by studying dynamics induced by cross-talk of biological systems.
 \end{remark}

\section{A lemma for Theorem~\ref{thm:endotactic=weakly-reversible}}

\begin{lemma}\label{le:A-endotactic}
Let $\cG=(\cV,\cE)$ be a first order reaction graph embedded in $\mathbb{N}^d_0$.  Let $w_+=\sum_{i\in \cI}e_i$ for some $\emptyset\neq\cI\subseteq[d]$ and $w_-=-w_+$. Assume $\cG^0\neq\emptyset$, $w_+\notin\mathsf{S}_{\cG}^{\perp}$, and $\cV_+=\{0\}\cup\{e_i\}_{i\in[d]}$. Then
\begin{equation}\label{eq:maximality-condition}
y\ \text{is}\ w_+\text{-maximal in}\ \cV_+ \Leftrightarrow y\in\{e_i\}_{i\in\cI};\ e_i\ \text{is}\ w_-\text{-maximal in}\ \cV_+ \Leftrightarrow y\in\{0\}\cup\{e_i\}_{i\in[d]\setminus\cI}
\end{equation}
Moreover,
\begin{enumerate}
\item[(i)] if $\cG^{\spadesuit}$ is $w_+$-strongly endotactic, then  $$\supp_{w_+}\cG=\supp_{w_+}\cG^{\spadesuit}\subseteq\{e_i\}_{i\in\cI}\ \ \text{and}\ \ \cE_{w_+}\setminus\{0\to y'\in\cE\}\subseteq \cE^{\spadesuit}_{w_+};$$
\item[(ii)] if $\cG^{\spadesuit}$ is $w_-$-strongly endotactic, then $$\supp_{w_-}\cG=\supp_{w_-}\cG^{\spadesuit}\subseteq\{0\}\cup\{e_i\}_{i\in[d]\setminus\cI}\ \ \text{and}\ \ \cE_{w_-}\setminus\{0\to y'\in\cE\}\subseteq \cE^{\spadesuit}_{w_-}.$$
\end{enumerate}
\end{lemma}
\begin{proof}
Note that \eqref{eq:maximality-condition} follows from $\cV_+=\{0\}\cup\{e_i\}_{i\in[d]}$ and the definition of $w_+$ and $w_-$. For $u\in\mathbb{R}^d$, let $\cE^{\spadesuit}_u$ abbreviate $(\cE^{\spadesuit})_u$.
Then $\cE^{\spadesuit}_u=\emptyset$ implies $\cE_u=\emptyset$, since $\cE^{\spadesuit}$ and $\cE^{\spadesuit}$ may only differ by zeroth order reactions, and reaction vectors of the zeroth order reactions in $\cE$ are linear combinations of those of the zeroth order reactions in $\cE^{\spadesuit}$:
\begin{equation}\label{eq:linear-dependence}
y'\in\spa\{z\colon 0\ce{->}z\in \cE^{\spadesuit}\},\quad \forall 0\ce{->}y'\in\cE\setminus \cE^{\spadesuit}
\end{equation} Since $w_+\notin\mathsf{S}_{\cG}^{\perp}$, we have $\cE_{w_+}\neq\emptyset$, which implies $\cE^{\spadesuit}_{w_+}\neq\emptyset$ by contraposition. Analogously, we can show that $\cE_{w_-}\neq\emptyset$ and $\cE^{\spadesuit}_{w_-}\neq\emptyset$. Next, we prove (i) and (ii).
\medskip

\noindent (i) Since $\cG^{\spadesuit}$ is $w_+$-strongly endotactic,  we have $\supp_{w_+}\cG^{\spadesuit}\neq\emptyset$ and every element in $\supp_{w_+}\cG^{\spadesuit}$ is $w_+$-maximal  in $\cV_{w,+}$. Then it follows from  \eqref{eq:maximality-condition} that $$\supp_{w_+}\cG^{\spadesuit}\subseteq\{e_i\}_{i\in\cI},$$ and hence $0\notin\supp_{w_+}\cG^{\spadesuit}$. By the  construction of $\cG^{\spadesuit}$,  $\cG$ and $\cG^{\spadesuit}$ share the same subset of first order reactions: $\cG^*=(\cG^{\spadesuit})^*$. Hence $$\supp_{w_+}\cG=\supp_{w_+}\cG^{\spadesuit};\quad \cE_{w_+}\setminus\{0\to y'\in\cE\}\subseteq \cE^{\spadesuit}_{w_+}$$
\medskip

\noindent(ii) Similar to (i),  $w_-$-strong endotacticity  of $\cG^{\spadesuit}$ yields $$\emptyset\neq\supp_{w_-}\cG^{\spadesuit}\subseteq\{0\}\cup\{e_i\}_{i\in[d]\setminus\cI}$$ Note that $0\in \cV^{\spadesuit}_+=\cV_+$.

If $0\in\supp_{w_-}\cG^{\spadesuit}$, then there exists $0\to e_j\in \cE^{\spadesuit}_{w_-}$ such that $0>_{w_-}e_j$ due to $w_-$-endotacticity of $\cG^{\spadesuit}$. This yields $j\in\supp w_-$. By the construction of $\cG^{\spadesuit}$, there exists $0\to y'\in\cE$ such that $y_j'>0$. In the light of \eqref{eq:maximality-condition} and $w_-\le0$, this yields that $$0\to y'\in\cE_{w_-},\quad 0\in\supp_{w_-}\cG$$ Analogously, since $\cG^*=(\cG^{\spadesuit})^*$, we have $\cG^*\cap\cE_{w_-}=(\cG^{\spadesuit})^*\cap \cE^{\spadesuit}_{w_-}$, and hence \[\supp_{w_-}\cG=\supp_{w_-}\cG^{\spadesuit}\subseteq\{0\}\cup\{e_i\}_{i\in[d]\setminus\cI}\]
and \begin{equation}\label{eq:essential-reactions}
\cE_{w_-}\setminus\{0\to y'\in\cE\}\subseteq \cE^{\spadesuit}_{w_-}
\end{equation}

If $0\notin\supp_{w_-}\cG^{\spadesuit}$, then due to \eqref{eq:maximality-condition} we conclude by contraposition that $0\notin \cV^{\spadesuit}_{w_-,+}$, i.e., $$(z'-0)\cdot w_-^T=0,\quad 0\to z'\in \cE^{\spadesuit},$$ 
which also yields from \eqref{eq:linear-dependence} that $$\{0\to z'\in\cE\}\cap\cE_{w_-}=\emptyset$$ Hence $0\notin\supp_{w_-}\cG$. In this case, the conclusion also holds with \eqref{eq:essential-reactions} and
$$\supp_{w_-}\cG=\supp_{w_-}\cG^{\spadesuit}\subseteq\{e_i\}_{i\in[d]\setminus\cI}$$
\end{proof}

\section{Formula for positive equilibria}\label{subsec:formular_for_positive_equilibria}

Let $\cG$ be a first order endotactic reaction graph. It is easy to observe that \emph{the multiplicity of the zero eigenvalue equals the number $k$ of strongly connected components $\cG^i$ of $\cG$ not containing the zero complex}. On the one hand, the average flux matrix confined to the sub reaction graph $\cG^0$ is non-singular by Proposition~\ref{prop:endotactic-Hurwitz}. On the other hand, as will be seen below (Theorem~\ref{thm:exisitence-of-positive-equilibrium}), the ODE for the sub reaction graph $\cG^{\bullet}$ is decomposed blockwise into $k$ ODEs, each of which models the mass-action system $\cG^i$. Each  of the $k$ ODEs is indeed a \emph{chemical master equation} (CME) associated with an irreducible CTMC on a finite state space\---the standard simplex in $\mathbb{R}^{d_i}$, where $d_i$ is the number of species of the strongly connected component $\cG^i$.

Due to Theorem~\ref{thm:A-endotactic=endotactic}, let  $\mathcal{G}^{\bullet}=\cup_{i=1}^k\mathcal{G}^i$ consist of $k\in\mathbb{N}_0$ strongly connected components $\cG^i=(\cV^i,\cE^i)$, where $\cV^i=\{e_\ell\}_{\ell\in\cI_i}$ for $\cI_i\subseteq[d]_0$ and $i\in[k]_0$; by convention $e_0=0\in\cV^0$ if $\cG^0\neq\emptyset$, and $k=0$ when $\cG=\cG^0$. Let $n_i=\#\cI_i$ for $i\in[k]_0$. Let $c_{\ell}$ be the sum of weights of all \emph{spanning trees} of $\cG^i$ rooted at a vertex $e_{\ell}\in\cV^i$ for $i\in[k]_0$.
For $n\in\mathbb{N}$ and $a\in\mathbb{R}_+$, let $\Delta_n=\{y\in\mathbb{R}^n_+\colon \|y\|_1=1\}$ be the $(n-1)$-dimensional simplex of $\mathbb{R}^n_+$ and $a\Delta_n=\{ay\colon y\in\Delta_{n}\}$ a \emph{scaled} simplex, where by convention,  $0\Delta_n=\{0\}\subseteq\mathbb{R}^n_+$ is a \emph{degenerate}  scaled simplex.

If $n_0<d$, for every $s=(s_1,s_2,\ldots,s_k)\in\mathbb{R}^{d-n_0}_{+}$ with $s_{i}\in\mathbb{R}^{n_i}_{+}$,  $i\in[k]$, let
 $$\Gamma_s\coloneqq\Bigl\{y\in\mathbb{R}^d_+
\colon \sum_{\ell\in\cI_i}y_{\ell}=\|s_i\|_1,\ i\in[k]\Bigr\}$$
and define $x_*^{(s)}=(x_{*,1}^{(s)},\ldots,x_{*,d}^{(s)})\in\mathbb{R}^d_{+}$ by
\begin{equation}\label{eq:expression-of-positive-equilibrium}
x_{*,\ell}^{(s)}=\sum_{i=1}^ks_i\frac{c_{\ell}}{\sum_{j\in\cI_i}c_j}\mathbbm{1}_{\cI_i}(\ell)+ \frac{c_{\ell}}{c_0}\mathbbm{1}_{\cI_0}(\ell)\mathbbm{1}_{\cV}(0),\quad \ell\in[d]
\end{equation}
It is straightforward to verify that $x_*^{(s)}\in\Gamma_s$. In particular, if $n_0=0$, then $$\Gamma_s=\oplus_{i=1}^k\|s_i\|_1\Delta_{n_i};$$ 
if $n_0=d$,  let $\Gamma_{\emptyset}=\mathbb{R}^d_+$ and $x_*^{(\emptyset)}=(x_{*,1}^{(\emptyset)},\ldots,x_{*,d}^{(\emptyset)})\in\mathbb{R}^d_{++}$ with
\begin{equation}\label{eq:expression-of-positive-equilibrium-non-singular}
x_{*,\ell}^{(\emptyset)}=\frac{c_{\ell}}{c_0},\quad \ell\in[d]\end{equation}

We first represent the unique equilibrium of \eqref{eq:ODE-linear} when $\cG=\cG^0$.
\begin{lemma}\label{le:positivity-of-equilibrium}
Let $\cG$ be a first order endotactic mass-action system. Assume $\cG=\cG^0$. Let $A$ be its average flux matrix and $b$ the influx vector. Then $\Gamma_{\emptyset}=\mathbb{R}^d_+$ is the unique stoichiometric compatibility class of $\cG$, and $x^{(\emptyset)}_*=b(-A)^{-1}$
 is the unique equilibrium of $\cG$ in $\Gamma_{\emptyset}$ which is positive.
\end{lemma}
\begin{proof}

Since $\cG=\cG^0$, by Lemma~\ref{le:path-back-to-0}, we have $\cV_+=\{e_j\}_{j\in[d]_0}$, $\mathsf{S}_{\cG}=\mathbb{R}^d$, and $\Gamma_{\emptyset}=\mathbb{R}^d_+$ is the unique stoichiometric compatibility class of $\cG$. Since $\cG$ is endotactic, by Proposition~\ref{prop:endotactic-Hurwitz}, $A$ is non-singular and $x=b(-A)^{-1}$ is the unique equilibrium of $\cG$.

By Theorem~\ref{thm:endotactic=weakly-reversible}, $\cG^{\spadesuit}$ is a realization of $\cG$. We assume w.l.o.g. that $\cG=\cG^{\spadesuit}$, i.e., $\cG$ is monomolecular and strongly connected. Let $\mathsf{L}(\mathcal{G})$ denote the \emph{Laplacian} of $\mathcal{G}$ \cite{K97}. For simplicity, we denote $\kappa_{e_i\to e_j}$ by $\kappa_{ij}$, for $i,j\in[d]_0$. By Proposition~\ref{prop:endotactic-Hurwitz}, $$(\mathsf{L}(\cG))_{ij}=\begin{cases}-\kappa_{ij},& \text{if}\ i\neq j, \\
\sum_{\ell\neq i}\kappa_{i\ell},& \text{if}\ i=j,\end{cases}\quad i,j\in[d]_0$$ Since $\mathcal{G}$ is strongly connected, by \emph{Kirchhoff Matrix Tree Theorem} \cite{K1847} for weighted directed graphs (also called \emph{Tutte's Theorem} \cite{T48})), for $\ell\in[d]_0$, \emph{$c_{\ell}>0$ is the cofactor of the diagonal element $(\mathsf{L}(\cG))_{\ell \ell}$}, and $(c_0,\ldots,c_d)$ is the unique left eigenvector of $\mathsf{L}(\cG)$ w.r.t. the simple eigenvalue $0$ up to a scalar. Note that $$A=(-(\mathsf{L}(\cG))_{ij})_{i,j\in[d]},\quad b=-((\mathsf{L}(\cG))_{01},\ldots,(\mathsf{L}(\cG))_{0d})$$ Hence \emph{$xA+b=0$ if and only if $[x\ 1]\in\mathbb{R}^{d+1}$ is a left eigenvector of $\mathsf{L}(\cG)$ w.r.t. the simple eigenvalue $0$}. This implies that  the equilibrium $b(-A)^{-1}$ coincides with \eqref{eq:expression-of-positive-equilibrium-non-singular} and hence is positive.
\end{proof}

\begin{theorem}\label{thm:exisitence-of-positive-equilibrium}
Let $\cG$ be a first order endotactic mass-action system. There exists a unique equilibrium in each stoichiometric compatibility class. More precisely,
\begin{enumerate}
\item[(i)]  If $n_0=d$, then $\Gamma_{\emptyset}$ is the unique  stoichiometric compatibility class of  $\cG$ with a unique positive equilibrium $x_*^{(\emptyset)}$.
\item[(ii)] If $n_0<d$, then $\Gamma_s$ is a  stoichiometric compatibility class of  $\cG$ for every $s=(s_1,s_2,\ldots,$ $s_k)\in\mathbb{R}^{d-n_0}_{+}$ with a unique equilibrium $x_*^{(s)}$, and in particular the interior of $\Gamma_s$ is a positive stoichiometric compatibility class of  $\cG$ containing $x_*^{(s)}>0$ if and only if $s=(s_1,s_2,\ldots,s_k)\in\mathbb{R}^{d-n_0}_{++}$.
\end{enumerate}
 \end{theorem}
\begin{proof}

By Theorem~\ref{thm:A-endotactic=endotactic}, $\cG^0$ and $\cG^{\bullet}$ are sub reaction graphs of disjoint sets of species, and $\cG^{\bullet}$ is 
 weakly reversible with strongly connected components of pairwise disjoint sets of species. Hence $A$ is block diagonal, and it suffices to prove (i) when $\cG=\cG^0$; 
  and a special case of (ii):   
when  $\cG=\cG^{\bullet}$ is strongly connected with $n_0=0$.
\medskip

\noindent(i) It follows immediately from Lemma~\ref{le:positivity-of-equilibrium}.
\medskip

\noindent(ii) Assume $\cG=\cG^{\bullet}$. In this case, $\mathsf{L}(\cG)=-A$ is the Laplacian of $\cG$ and $\Gamma_s=\|s\|_1\Delta_d$ for every $s\in\mathbb{R}^d_+$. In particular, 
$\Gamma_0=\{0\}$ consisting of the zero equilibrium of $\cG$ is a (degenerate) stoichiometric compatibility class.   Next, we consider the case when $s\neq0$. 
Let $s\in\mathbb{R}^d_+\setminus\{0\}$. Note that  $b=0$ as $0\notin\cV_+$. 
Using a similar argument as in the proof of Lemma~\ref{le:positivity-of-equilibrium} based on the Kirchhoff Matrix Tree Theorem, one can show that $x_*^{(s)}$ given in \eqref{eq:expression-of-positive-equilibrium}  is  the unique equilibrium of $\cG$ in $\Gamma_s$.

 Positivity of the stoichiometric compatibility class simply  follows from the definition of $\Gamma_s$.
\end{proof}

\begin{remark}\label{re:positive-equilibrium}
\begin{enumerate}
\item[(i)] Below is a direct implication of Theorem~\ref{thm:exisitence-of-positive-equilibrium}:  \emph{The minimal order for an endotactic mass-action system to have multiple positive equilibria in a positive stoichiometric compatibility class is 2}. Consider the bimolecular Edelstein network:
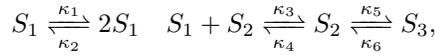
\[S_1\ce{<=>[\kappa_1][\kappa_2]}2S_1\quad S_1+S_2\ce{<=>[\kappa_3][\kappa_4]}S_2\ce{<=>[\kappa_5][\kappa_6]}S_3,\]
which is (weakly) reversible, and hence is endotactic. It is known that for certain choices of the reaction rate constants, this mass-action system is \emph{bistable} with three positive equilibria in a positive stoichiometric compatibility class  \cite[Example~3.C.3]{F79}. Indeed, higher order endotactic  or weakly reversible mass-action systems ( of positive deficiency)  may even have infinitely many positive equilibria \cite{BCY20,KD23}.
\item[(ii)] By Theorem~\ref{thm:endotactic=weakly-reversible}, applying Deficiency Zero Theorem \cite[Theorem~6.1.1]{F87} (see also \cite{HJ72,FH77}) to $\cG^{\spadesuit}$ also yields the existence of a unique positive equilibrium in each positive  stoichiometric compatibility class. However, as remarked in \cite{F87}, it \emph{cannot} exclude the existence of boundary equilibria.
\item[(iii)] Kirchhoff Matrix Tree Theorem has been commonly used in the literature to obtain formula of positive equilibria of reaction systems, e.g., in \cite{CDSS09}.
%\item[\red{(iv)}] Despite the result on positivity of equilibrium seems classical, particularly in the light of the literature on compartmental systems \cite{HCH10,FR11,A83} (also called ``compartmental models'' \cite{HCH10} and ``positive linear systems" \cite{FR11}), it indeed is \emph{not}. Although the associated ODE is a positive linear system \cite[Definition~2, Chapter~2]{FR11}, by \cite[Theorem~2, Chapter~2]{FR11}, even based on the asymptotic stability given in Theorem~\ref{thm:deterministic-stability-endotactic}, we cannot infer the positivity of the unique equilibrium since $b$ may not be strictly positive as required in \cite[Theorem~19,Theorem~20, Chapter~8]{FR11}. 
\end{enumerate}
\end{remark}

From Theorem~\ref{thm:exisitence-of-positive-equilibrium}, $x_*^{(s)}$ is \emph{not} positive if and only if $\Gamma_s$ has an empty interior, precisely when $s$ has zero entries.
 For instance, when concentration of all species in one of the strongly connected components without the zero complex is set to be zero, then $x_*^{(s)}\not>0$.

\begin{example}\label{ex:degenerate-class}
Revisit Example~\ref{ex:Introductory-example} 
 with specific reaction rate constants: 
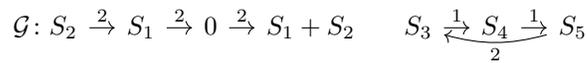
\begin{center}\begin{tikzcd}[row sep=1em, column sep = 1em]
\cG\colon S_2\rar{2}&S_1\rar{2}&0\rar{2}&S_1+S_2&S_3\rar{1}&S_4\rar{1}&S_5\arrow[bend left=15]{ll}{2}
\end{tikzcd}
\end{center}
We have $n_0=2,\ n_1=3$. By Theorem~\ref{thm:exisitence-of-positive-equilibrium},  $x_*^{((0,0,0))}=(2,1,0,0,0)$ is the unique equilibrium of $\cG$ in the stoichiometric compatibility class $\Gamma_{(0,0,0)}=\mathbb{R}^2_+\times\{(0,0,0)\}$. 
\end{example}
For first order mass-action systems, despite endotacticity implies (1) \emph{the existence of a positive equilibrium} in each positive stoichiometric compatibility class as well as (2) the 
\emph{average flux matrix $A$ is semi-stable}, conversely,
the two properties (1) and (2) together may \emph{not} yield that the reaction system has an endotactic MAK realization.
\begin{example}
Consider the following mass-action system:
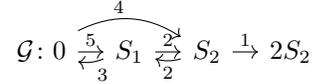
\begin{center}\begin{tikzcd}[row sep=1em, column sep = 1em]
\cG\colon 0\arrow[bend left=30]{rr}{\hspace{-1em}4 }\arrow[rightharpoonup]{r}{5}&S_1\arrow[yshift=-.08cm,rightharpoonup]{l}{3}\arrow[rightharpoonup]{r}{2}&S_2 \arrow[yshift=-.08cm,rightharpoonup]{l}{2}\rar{1}&2S_2
\end{tikzcd}
\end{center}
The average flux matrix and the influx vector in \eqref{eq:ODE-linear} associated with $\cG$ are given by \[A=\begin{bmatrix}
-5&2\\
2&-1
\end{bmatrix},\quad b=[5\ 4]\]
It is straightforward to verify that $A$ is Hurwitz, and $x_*=[13 \ 30]$ is the unique equilibrium of $\cG$ in the unique stoichiometric compatibility class $\mathbb{R}^2_+$. Nevertheless, by the proof of  Proposition~\ref{prop:endotactic-Hurwitz}, the ODE fails to have an endotactic first order mass-action system realization since
$\mathbf{1}A^T\not\le0$.
\end{example}

Despite $(-A)^{-1}$ exists and is non-negative for a Metzler Hurwitz stable matrix $A$ (c.f.,  \cite[Theorem~2.5.3]{HJ94} or \cite[Chapter 6, Theorem~2.3]{BP94}), $x_*=b(-A)^{-1}\ge0$ may \emph{not} be strictly positive if the first order reaction network is \emph{not} endotactic.
\begin{example}
Consider the mass-action system
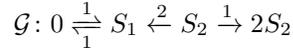
\[\cG\colon 0\ce{<=>[1][1]}S_1\ce{<-[2]}S_2\ce{->[1]}2S_2\]
The associated ODE has the corresponding average flux matrix and the influx vector \[A=\begin{bmatrix}
-1&0\\
2&-1\end{bmatrix},\quad b=[1\ 0]\]
It is easy to verify that $A$ is Metzler and Hurwitz stable while  $x_*=b(-A)^{-1}=[1\ 0]$, the unique equilibrium  of $\cG$,   is \emph{not} positive.
\end{example}

\section{Diagonally dominant matrix}
Recall that a matrix $A$ is \emph{diagonally dominant} if for each $i\in[d]$,
\begin{equation*}
|a_{ii}|\ge\sum_{j\neq i}|a_{ij}|
\end{equation*}
In particular, a row $i$ is called \emph{strictly diagonally dominant} (SDD) if \begin{equation*}
|a_{ii}|>\sum_{j\neq i}|a_{ij}|
\end{equation*}
A diagonally dominant matrix $A$ is called \emph{weakly chained diagonally dominant} (WCDD) if for each non-SDD row $j$, there exists a path in the associated directed graph of $A$ from the vertex $j$ to a vertex $i$, where row $i$ is SDD \cite{BH64,SC74,AF16}. 
\section{A proposition on exponential asymptotic stability}\label{sect:GAS}
 
It is noteworthy that GAS of a positive equilibrium implies permanence of the reaction system. Let us first recall the definition of permanence.

\begin{definition}\label{def:persistence+boundedness+permanence}
Let $\cG$ be a reaction system of $d$ species in terms of the ODE \eqref{eq:ODE}. {Let $\Gamma$ be a stoichiometric compatibility class of $\cG$.}
\begin{enumerate}
\item[$\bullet$] $\cG$ is \emph{persistent} on $\Gamma$  if, regardless of the initial condition subject to the interior of $\Gamma$, its solution $x(t)$ satisfies
\[\min_{1\le i\le d}\liminf_{t\to\infty}|x_i(t)|>0\] 
{Moreover, $\cG$ is \emph{uniformly persistent} on $\Gamma$ if there exists $\epsilon>0$ such that 
\[\min_{1\le i\le d}\liminf_{t\to\infty}|x_i(t)|\ge \epsilon\]
regardless of the initial condition in the interior of $\Gamma$.}
\item[$\bullet$] $\cG$ \emph{has bounded trajectories}  if, regardless of the initial condition, its solution $x(t)$ satisfies
\[\limsup_{t\to\infty}\|x(t)\|_1<\infty\]
{Moreover, $\cG$ is \emph{has uniformly bounded trajectories} on $\Gamma$ if there exists $\epsilon>0$ such that 
\[\min_{1\le i\le d}\limsup_{t\to\infty}|x_i(t)|\le 1/\epsilon\]
regardless of the initial condition in the interior of $\Gamma$.}
\item[$\bullet$] {$\cG$ is \emph{permanent} on $\Gamma$ if $\cG$ has a compact global attractor which is a subset of the interior of $\Gamma$.}
\end{enumerate}
\end{definition}
Despite persistence with boundedness (of trajectories) is weaker than permanence,  \emph{uniform} persistence with  \emph{uniformly} boundedness is equivalent to permanence of a reaction system.\footnote{Indeed, uniform boundedness yields the existence of a compact global attractor $K$ in $\Gamma$, and uniform persistence further yields a compact subset of $K$ in the interior of $\Gamma$ is also a global attractor, which implies permanence.}

\begin{proposition}\label{thm:linear-ode}
Let $\cG$ be a reaction system of $d$ species. Assume the evolution of concentration of species of $\cG$ %is governed by
follows \eqref{eq:ODE-linear} with a Hurwitz stable matrix $A$.
Then $\cG$ has a unique non-negative equilibrium $x_*=b(-A)^{-1}$ which is globally asymptotically stable in $\mathbb{R}^d_+$. More precisely, there exists a polynomial $g(t)$ of degree $\le d-1$  depending on the initial concentration $x_0$ such that
\begin{equation}\label{eq:sharp-exponential}
\|x(t)-x_*\|_1\le g(t)e^{-r(A) t}
\end{equation}
\end{proposition}

\begin{proof}
Since $A$ is Metzler and Hurwitz stable, we have $A$ is non-singular and $A^{-1}$ is non-negative \cite[Theorem~2.5.3]{HJ94} (see also \cite[Chapter 6, Theorem~2.3]{BP94}). Hence $x_*=b(-A)^{-1}\ge0$ since $b\ge0$. Then the global exponential convergence in terms of \eqref{eq:sharp-exponential} follows from the Fundamental Theorem for linear autonomous ODEs \cite[Chapter~1]{P13}.
\end{proof}

\bibliographystyle{plain}
{\small\bibliography{references}}
\end{document}